\documentclass[12pt,reqno]{amsart}%
\usepackage{amsmath, amsfonts, amssymb, amsthm, amscd, amsbsy}
\usepackage[usenames,dvipsnames,svgnames,x11names,hyperref]{xcolor}
\usepackage[tmargin=0.8in,bmargin=0.8in,lmargin=0.9in,rmargin=0.9in]{geometry}
\usepackage{enumerate}
\usepackage{xcolor}
\usepackage[pagebackref]{hyperref}%
\usepackage{amsmath}%
\usepackage{bm}
\usepackage{amsfonts}%
\usepackage{amssymb}
\usepackage{mathrsfs}
\usepackage[numbers,sort]{natbib}

\providecommand{\U}[1]{\protect\rule{.1in}{.1in}}
\hypersetup{backref=true,
  pagebackref=true,
  hyperindex=true,
  colorlinks=true,
  breaklinks=true,
  urlcolor=NavyBlue,
  linkcolor=Fuchsia,
  bookmarks=true,
  bookmarksopen=false,
  filecolor=black,
  citecolor=ForestGreen,
  linkbordercolor=red
}

\newtheorem{theorem}{Theorem}[section]
\theoremstyle{plain}

\newtheorem{corollary}[theorem]{Corollary}

\newtheorem{lemma}[theorem]{Lemma}

\newtheorem{proposition}[theorem]{Proposition}
\newtheorem{remark}[theorem]{Remark}

\numberwithin{equation}{section}

\theoremstyle{definition}

\allowdisplaybreaks

\def\a{\alpha}
\def\b{\beta}
\def\d{\delta}

\def\p{\partial}
\def\g{\gamma}
\def\grad{\nabla}
\def\oo{\infty}

\def\R{\mathbb{R}}

\def\B{\mathbb{B}}
\def\C{\mathbb{C}}
\def\H{\mathbb{H}}
\def\N{\mathbb{N}}
\def\S{\mathbb{S}}

\def\mcC{\mathcal{C}}
\def\mcU{\mathcal{U}}
\def\mcH{\mathcal{H}}

\def\mcP{\mathcal{P}}

\def\Re{\operatorname{Re}}

\def\ceven{C_{\operatorname{even}}^{\oo}(\overline{\mcU^{n+1}})}
\def\CHspace{\mcC^{2\g}(\mcU^{n+1}) \cap \dot H^{k,\g}(\mcU^{n+1})}
\def\CHspaceball{\mcC^{2\g}(\B_{\C}^{n+1}) \cap \dot H^{k,\g}(\B_{\C}^{n+1})}

\begin{document}
\title[Boundary Operators and CR Sobolev Trace Inequalities]{Conformally Covariant Boundary Operators and Sharp Higher Order CR Sobolev Trace Inequalities on the Siegel Domain and Complex Ball}
\author{Joshua Flynn}
\address{Joshua Flynn: CRM/ISM and McGill University\\
Montr\'{e}al, QC H3A0G4, Canada}
\email{joshua.flynn@mcgill.ca}
\author{Guozhen Lu }
\address{Guozhen Lu: Department of Mathematics\\
University of Connecticut\\
Storrs, CT 06269, USA}
\email{guozhen.lu@uconn.edu}
\author{Qiaohua Yang}
\address{School of Mathematics and Statistics, Wuhan University, Wuhan, 430072, People's  Republic of China}
\email{qhyang.math@whu.edu.cn}

\thanks{The first two authors were partially supported by a grant from the Simons Foundation. The third author was  partially supported by the National Natural Science Foundation of China (No.12071353)}
  \textbf{ }
\subjclass[2000]{Primary 43A85, 43A90, 42B35, 42B15, 42B37, 35J08; }

\begin{abstract}
We first introduce    an appropriate family of conformally covariant boundary operators associated to  the Siegel domain $\mcU^{n+1}$ with the Heisenberg group $\H^{n}$ as its boundary  and the complex ball $\B_{\C}^{n+1}$ with the complex sphere $S^{2n+1}$ as its boundary. We provide the explicit formulas of these conformally covariant boundary operators. Second,
 we establish all higher order extension theorems  of Caffarelli-Silvestre type for the Siegel domain and complex ball. Third, we prove all higher order CR Sobolev trace inequalities  for the Siegel domain $\mcU^{n+1}$  and the complex ball $\B_{\C}^{n+1}$.
In particular, we generalize the  Sobolev trace inequality in the CR setting by Frank-Gonz\'alez-Monticelli-Tan
 to the case for all $\gamma\in (0, n+1)\backslash \mathbb{N}$. The family of higher order conformally covariant boundary operators we define are naturally intrinsic to the higher order Sobolev trace inequalities on both the Siegel domain $\mcU^{n+1}$ and complex ball $\B_{\C}^{n+1}$. Finally,
we give  an explicit solution to the scattering problem on the complex hyperbolic ball. More precisely, we obtain an integral representation and an expansion in terms of special functions for the solution to the scattering problem.

\end{abstract}

\keywords{}
\maketitle
\tableofcontents

\section{Introduction and statement of results}

Caffarelli and Silvestre showed in \cite{Caff1} that, for $\g \in (0,1)$, the fractional Laplacian $(-\Delta_{x})^{\g}$ on $\R^{n}$ may be recovered via a Dirichlet-to-Neumann map associated to the weighted Laplacian $ \Delta + (1-2\g) y^{-1} \p_{y}$ on the halfspace $\R_{+}^{n+1} := \R^{n} \times (0,\oo)$.
More precisely, if $U$ is a solution to
\begin{equation}
  \begin{cases}
     (\Delta + (1-2\g)y^{-1}\p_{y}) U = 0 & \text{ in }\R_{+}^{n+1}\\
    U = f &  \text{ on }\R^{n}
  \end{cases}
  \label{eq:2nd-order-dirichlet}
\end{equation}
then one has the Dirichlet-to-Neumann operator
\begin{equation}
  U \mapsto (-\Delta_{x})^{\g} f = -d_{\gamma}^{-1} \lim_{y \to 0 } y^{1-2\g}\p_{y}U,\;\;d_{\gamma}= 2^{1-2\gamma}\frac{\Gamma(1-\gamma)}{\Gamma(\gamma)},
  \label{eq:first-order-boundary-operator}
\end{equation}
thereby recovering $(-\Delta_{x})^{\g}$ on $\R^{n}$ via a ``boundary operator'' mapping functions on $\overline{\R_{+}^{n+1}}$ to functions on $\R^{n}$.
Related, by the Dirichlet principle, is the first order sharp Sobolev trace inequality
\begin{equation}
d_{\gamma}  \int_{\R^{n}} f (-\Delta_{x})^{\g}f dx \leq \int_{\R_{+}^{n+1}} |\grad U|^{2} y^{1-2\g} dxdy
  \label{eq:first-order-trace-inequality}
\end{equation}
where $f = U(\cdot, 0)$ and equality holds iff $U$ solves \eqref{eq:2nd-order-dirichlet}.

\smallskip

Given that $(-\Delta_{x})^{\g}$ on $\mathbb{R}^n$ may be defined for $\g \in (0,\oo)$, it is natural to investigate higher order analogues of Caffarelli-Silvestre's extension and Sobolev trace inequalities.
This was achieved by Case in \cite{MR4095805} by introducing suitable higher order boundary operators which map solutions to a higher order Dirichlet problem (generalizing \eqref{eq:2nd-order-dirichlet}) to fractional operators acting on the boundary data.
We should compare  these results with existing earlier ones in the literature.
For example, in \cite{MR2737789}, Chang-Gonz\'alez obtained through clever inductive arguments and Graham-Zworski scattering theory \cite{MR1965361} all of the higher order extension results.
In \cite{MR3592161,yang2013higher}, Chang-Yang established higher order extension results with corresponding sharp trace inequalities under  some Neumann-type boundary conditions in order for the Sobolev trace inequalities to hold.

It is well-known that $\R^{n}$ may be realized as the conformal infinity of real hyperbolic space.
Moreover, through the Graham-Zworski scattering theory \cite{MR1965361}, one may recover the fractional and integral order conformally covariant operators on the boundary of a conformally compactifiable Einstein manifold by studying a Poisson equation on the interior manifold.
As such, the fractional Laplacian on $\R^{n}$ is accessible through this scattering theory.
Pertinent to our results, there is also a geometric analogue of conformal geometry in the complex manifold setting, namely, CR geometry.
Indeed, just as $\R^{n}$ is a conformal infinity of real hyperbolic space and realized as the boundary of the halfspace $\R^{n+1}_{+}$, one may realize the Heisenberg group $\H^{n} = \C^{n} \times \R$ as the boundary of the Siegel domain $\mathcal{U}^{n+1} \cong \H^{n} \times (0,\oo) \subset \C^{n+1}$ and endow $\H^{n}$ with its natural contact structure.
Here, $\mcU^{n+1}$ may be equipped with a natural K\"ahler metric isometric to the Bergman metric on the complex ball $\B_{\C}^{n+1} \subset \C^{n+1}$, thereby realizing $\mcU^{n+1}$ as a halfspace model of the complex hyperbolic space $H_{\C}^{n+1}$ and establishing the CR equivalence of $\H^{n}$ and the CR sphere $S^{2n+1}= \p \B_{\C}^{n+1}$.
Importantly, $\mcU^{n+1}$ also admits a scattering theory (e.g., \cite{epstein,MR2472889,MR1965361, {MR0407320}, gover}) allowing one to obtain fractional and integer order CR covariant operators on $\H^{n}$ and $\mathbb{S}^{2n+1}$.
It is therefore a natural and geometrically significant problem to establish a Caffarelli-Silvestre extension theory for $\H^{n}$.
In \cite{Frank2}, Frank-Gonz\'alez-Monticelli-Tan established the first order extension theorem (i.e., for fractional powers in $(0,1)$) and corresponding first order CR Sobolev trace inequality for $\H^{n}$.
To recall their results and  state our results and for completeness, we begin by fixing notations and introducing relevant definitions.

The Siegel domain $\mathcal{U}^{n+1}\subset \mathbb{C}^{n+1}$ is defined as a superlevel set of the defining function
\begin{equation*}
  q(\zeta)=\textrm{Im}z_{n+1}-\sum_{j=1}^{n}|z_{j}|^{2},
\end{equation*}
namely
$$
\mathcal{U}^{n+1}:=\{\zeta=(z_{1},\cdots,z_{n},z_{n+1})=(z,z_{n+1})\in\mathbb{C}^{n}\times \mathbb{C}\;|\; q(\zeta)>0\}.
$$
Relative to the defining function $q$, $\mathcal{U}^{n+1}$ is a K\"ahler-Einstein manifold with K\"ahler form
\[
  \omega_{+} = - \frac{i}{2}\p \bar\p \log q = \frac{i}{2}\left( q^{-2}(4^{-1}dq^{2} + \theta^{2}) + q^{-1}(\d_{jk}\theta^{j} \wedge \theta^{\bar k}) \right),
\]
where $\theta^{j} = d z_{j}$ and $\theta^{\bar k} = d \bar z_{k}$.
Relative to the defining function $\rho = \sqrt{2 q}$, we may write the corresponding K\"ahler metric as
\begin{equation}
  g^{+} = \frac{1}{2}\left( \frac{d\rho^{2}}{\rho^{2}} + \frac{2 \d_{jk} \theta^{j} \times \theta^{\bar k}}{\rho^{2}} + \frac{4\theta^{2}}{\rho^{4}} \right).
  \label{eq:kahler-metric-on-siegel-domain}
\end{equation}
It is well-known that $\mcU^{n+1}$ equipped with $g^{+}$ is isometric to the complex hyperbolic space $H_{\C}^{n+1}$.
The volume form is (see e.g. \cite{gra1})
\begin{equation*}
  dV=\frac{1}{4q^{n+2}}dzdtdq=\frac{2^{n}}{\rho^{2n+3}}dzdtd\rho
\end{equation*}
and the Laplace-Beltrami operator on $(\mathcal{U}^{n+1},g^{+})$  is given by (see e.g. \cite{gra1})
\begin{equation}
\begin{aligned}
  \Delta_{\mathbb{B}}=&4q[q(\partial_{qq}+\partial_{tt})+\frac{1}{2}\Delta_{b} -n\partial_{q}]\\
  =&\rho^{2}(\partial_{\rho\rho}+\Delta_{b}+\rho^{2}\partial_{tt})-(2n+1)\rho\partial_{\rho},
\end{aligned}
  \label{eq:laplace-beltrami-operator-siegel-domain}
\end{equation}
where $\Delta_{b}$ is the sub-Laplacian on $\mathbb{H}^{n}$ (see (\ref{eq:heisenberg-group-sub-laplacian})). Recall that $\operatorname{spec}(-\Delta_{\B}) = [(n+1)^{2},\oo)$. 

As a hypersurface in $\C^{n+1}$, the boundary $\p \mcU^{n+1}$ has a natural CR structure which is CR equivalent to the CR sphere $\mathbb{S}^{2n+1} \subset \C^{n+1}$ and the Heisenberg group $\mathbb{H}^{n} \cong \C^{n} \times \R$.
We recall that $\mathbb{H}^{n}$ is a step 2 nilpotent Lie group with group structure
\begin{equation*}
  (z,t)(z',t')=(z+z',t+t'+2 \textrm{Im}(z,z')),
\end{equation*}
where $z,z'\in\C^{n}$ and $(z,z')=z_{1}\bar z_{1}' + \cdots z_{n}\bar z_{n}'$ is the Hermite inner product of $z$ and $z'$.
In the real coordinates given by $z_{j} = x_{j} + i y_{j}$, the left-invariant vector fields on $\mathbb{H}^{n}$ are given by
\begin{equation}
  \begin{split}
    X_{j}=&\frac{\partial}{\partial x_{j}}+ 2y_{j}\frac{\partial}{\partial t},\;\;j=1,\cdots,n,\\
    Y_{j}=&\frac{\partial}{\partial y_{j}}- 2x_{j}\frac{\partial}{\partial t},\;\;j=1,\cdots,n,\\
    T=&2\frac{\partial}{\partial t},
  \end{split}
  \label{eq:vector-fields-on-heisenberg-group}
\end{equation}
Note that the $2n+1$ vector fields $X_{1},\cdots, X_{n},Y_{1},\cdots, Y_{n}, T$ are a basis of left-invariant vector fields for the Lie algebra of $\mathbb{H}^{n}$.
Moreover, the sub-Laplacian on $\mathbb{H}^{n}$ is given by
\begin{equation}
  \label{eq:heisenberg-group-sub-laplacian}
  \Delta_{b}=\frac{1}{2}\sum^{n}_{j=1}(X^{2}_{j}+Y^{2}_{j}).
\end{equation}

The natural identification of $\p\mathcal{U}^{n+1}$ with $\mathbb{H}^{n}$ is given by the mapping
\[
  \mathbb{H}^{n} \ni (z,t) \mapsto (z,t+i|z|^{2}) \in \p\mathcal{U}^{n+1}.
\]
As a CR manifold, the standard contact form $\theta$ on $\mathbb{H}^{n}$ is given by
\begin{equation}
  \theta = \frac{1}{2} \left( dt + i \sum_{j=1}^{n}(z_{j} d\bar z_{j} - \bar z_{j} d z_{j})  \right).
  \label{eq:standard-contact-structure-on-H}
\end{equation}
Note that $T = 2 \p_{t}$ is the corresponding characteristic direction and that, if we write $Z_{j} = X_{j} + i Y_{j}$, $j = 1, \ldots, n$, then $\left\{ Z_{j}:j=1,\ldots,n \right\}$ spans horizontal bundle that gives $\mathbb{H}^{n}$ its CR structure.
Under the identification of $\p \mcU^{n+1}$ with $\mathbb{H}^{n}$ and given that $(\mcU^{n+1},g_{+})$ is isometric to the complex hyperbolic space $H_{\C}^{n+1}$, we may naturally realize $\mcU^{n+1} \cong \mathbb{H}^{n} \times (0,\oo)$ as the halfspace model of $H_{\C}^{n+1}$.

In \cite{Frank2}, Frank et al used the scattering theory on complex hyperbolic space to construct CR covariant operators of fractional order $\g \in (0,1)$ on $\mathbb{H}^{n}$.
In more details, they consider the extension problem
\begin{align}\label{1.6}
\left\{
  \begin{array}{ll}
    \partial_{\rho\rho}U+a\rho^{-1}\partial_{\rho}U+\rho^{2}\partial_{tt}U+\Delta_{b}U=0, & \hbox{in $\mathcal{U}^{n+1}\backsimeq\mathbb{H}^{n}\times (0,\infty)$;} \\
    U=f, & \hbox{on $\partial \mathcal{U}^{n+1}\backsimeq\mathbb{H}^{n}$}
  \end{array}
\right.
\end{align}
and show that
\begin{align}\label{1.7}
\frac{c_{\gamma}}{\gamma2^{1-\gamma}}\lim_{\rho\rightarrow0}\rho^{a}\partial_{\rho}U=(2|T|)^{\gamma}\frac{\Gamma(\frac{1+\gamma}{2}+\frac{-\Delta_{b}}{2|T|})}
{\Gamma(\frac{1-\gamma}{2}+\frac{-\Delta_{b}}{2|T|})}f:=P_{\gamma}f,
\end{align}
where $\gamma\in (0,1)$, $a=1-2\gamma$ and
\begin{align}\label{1.8}
c_{\gamma}=2^{\gamma}\frac{\Gamma(\gamma)}{\Gamma(-\gamma)}.
\end{align}
We remark that
\begin{align}\label{1.9}
P_{\gamma}=(2|T|)^{\gamma}\frac{\Gamma(\frac{1+\gamma}{2}+\frac{-\Delta_{b}}{2|T|})}
{\Gamma(\frac{1-\gamma}{2}+\frac{-\Delta_{b}}{2|T|})}, \;\;\gamma\in (0,\infty),
\end{align}
are known as the CR covariant operators of fractional order and agree with the intertwining operators on the CR sphere (see \cite{Branson1,Branson2,Graham1}).
We note that the extension problem (\ref{1.6}) differs in form that considered by Caffarelli and Silvestre \cite{Caff1} for the construction of the fractional Laplacian in the Euclidean case.
The essential difference is the appearance of the additional term $\rho^{2}\partial_{tt}U$ in the $t$-direction.

By using the extension problem (\ref{1.6}), Frank et al \cite{Frank2} obtained the following sharp Sobolev trace inequality:
\begin{theorem}[Frank-Gonz\'alez-Monticelli-Tan]\label{Frank-Gonzalez-Monticelli-Tan}
Let $\gamma\in (0,1)$. Then there exists a unique linear bounded operator $\mathcal{T}: \dot{H}^{1,\gamma}(\mathcal{U}^{n+1})\rightarrow\dot{S}^{\gamma}(\mathbb{H}^{n})$
such that $\mathcal{T}(U)=U(\cdot,0)$ for all $U\in C_{0}^{\infty}(\mathcal{U}^{n+1})$, where $\dot{H}^{1,\gamma}(\mathcal{U}^{n+1})$ and $\dot{S}^{\gamma}(\mathbb{H}^{n})$ are the Sobolev spaces on $\mathcal{U}^{n+1}$ and $\mathbb{H}^{n}$, respectively.
Moreover, for any $U\in \dot{H}^{1,\gamma}(\mathcal{U}^{n+1})$ one has
\begin{equation}\label{1.12}
  \begin{split}
 &\int_{\mathcal{U}^{n+1}}\left[|\partial_{\rho}U|^{2}+\rho^{2}|\partial_{t}U|^{2}+\frac{1}{2}\sum_{j=1}^{n}(|X_{j}U|^{2}+|Y_{j}U|^{2})\right]\rho^{1-2\gamma}dxdydtd\rho\\
\geq&2^{1-2\gamma}\gamma\frac{\Gamma(1-\gamma)}{\Gamma(1+\gamma)}\int_{\mathbb{H}^{n}}\mathcal{T}(U)P_{\gamma}(\mathcal{T}(U))dxdydt.
  \end{split}
\end{equation}
Equality is attained if and only if $U$ is the unique solution of the extension
problem (\ref{1.6}) with some $f\in \dot{S}^{\gamma}(\mathbb{H}^{n})$.
\end{theorem}

\smallskip

One of the main purposes of this paper is to establish all higher order extension theorems (see Theorems \ref{thm:boundary-to-fractional-operator} and \ref{th1.10}) and higher order CR Sobolev trace inequalities (see Theorems \ref{th1.5} and \ref{th1.11}) for the Siegel domain $\mcU^{n+1}$ with boundary $\H^{n}$ and the complex ball $\B_{\C}^{n+1}$ with boundary $S^{2n+1}$.
In particular, we generalize Theorem \ref{Frank-Gonzalez-Monticelli-Tan} to the case for all $\gamma\in (0, n+1)\backslash \mathbb{N}$.
(see the recent work  \cite{Gunhee} for $\gamma\in (0, 2)\backslash \mathbb{N}$.)
Therefore, we provide a complete CR Caffarelli-Silvestre-type extension theory for the Siegel domain and complex ball.
Another purpose of this paper is to identify a class of  higher order conformally covariant boundary operators on the Siegel domain and complex ball.
This family of higher order conformally covariant boundary operators are naturally intrinsic to the higher order Sobolev trace inequalities on both the Siegel domain $\mcU^{n+1}$ and complex ball $\B_{\C}^{n+1}$.
In particular, our boundary operators generalize \eqref{1.7} to all relevant higher orders.
Another main result of this paper is providing an explicit solution to the scattering problem on the complex hyperbolic ball.
Before we state our main results, we recall some motivation and background in this direction following the exposition in \cite{LuYangpaper}.

\smallskip

Let $\mathbb{B}_{\mathbb{C}}^{n+1}$ denote the unit ball in $\mathbb{C}^{n+1}$ and centered at the origin.
 In \cite{ge},  Geller
introduced a family of second order degenerate elliptic
operators arising in several complex variables
  \begin{equation*}
  \begin{split}
\Delta_{\alpha,\beta}=4(1-|z|^{2})\left\{\sum^{n+1}_{j=1}\sum^{n+1}_{k=1}(\delta_{j,k}-z_{j}\overline{z}_{k})\frac{\partial^{2}}{\partial z_{j}
\partial\overline{z}_{k}}+\alpha R+\beta\bar{R}-\alpha\beta\right\},
\end{split}
\end{equation*}
where  $\delta_{j,k}$ denotes the Kronecker symbol and
  \begin{equation*}
  \begin{split}
R=\sum^{n+1}_{j=1}z_{j}\frac{\partial}{\partial z_{j}},\;\;\overline{R}=\sum^{n+1}_{j=1}\overline{z}_{j}\frac{\partial}{\partial \overline{z}_{j}}.
\end{split}
\end{equation*}
When $\alpha=\beta=0$, $\Delta_{0,0}$ is then the invariant Laplacian or Laplace-Beltrami
operator for the Bergman metric on $\mathbb{B}_{\mathbb{C}}^{n+1}$.
Set
  \begin{equation*}
  \begin{split}
\Delta'_{\alpha,\beta}=\frac{1}{4(1-|z|^{2})}\Delta_{\alpha,\beta}
=&\sum^{n+1}_{j=1}\sum^{n+1}_{k=1}(\delta_{j,k}-z_{j}\overline{z}_{k})\frac{\partial^{2}}{\partial z_{j}
\partial\overline{z}_{k}}+\alpha R+\beta\bar{R}-\alpha\beta.
\end{split}
\end{equation*}
The above family of operators $\Delta_{\alpha,\beta}$ and $\Delta'_{\alpha,\beta}$ have been considered by many
authors. For example, Geller \cite{ge} introduced such operators  to discuss the $H^{p}$ theory of Hardy spaces on the
Heisenberg group and
 Graham \cite{gra1,gra2} has given a precise description of the
associated Dirichlet problem.  We   remark that such operators  are closely related to the
invariant differential operator on the line bundle  of $SU(1, n+1)/S(U(n+1)\times U(1))$ associated with the one-dimensional representation
of $S(U(n+1)\times U(1))$
 (see \cite{sh}).

\medskip

 There are analogous Geller's type operators on the Siegel domain $
\mathcal{U}^{n+1}$ (see Section 2.2).
Let $t=\textrm{Re}z_{n+1}$ and $\mathcal{L}_{0}=-\frac{1}{2}\Delta_b$ (see (\ref{eq:heisenberg-group-sub-laplacian})) be the Folland-Stein operator on the Heisenberg group.
The analogous operators  are then
\begin{equation*}
  P_{\alpha,\beta}=4q[q(\partial_{qq}+\partial_{t}^{2})-\mathcal{L}_{0}-(n-1+\alpha+\beta)\partial_{q}
  -i(\alpha-\beta)\partial_{t}].
\end{equation*}
If $\alpha=\beta=0$, $P_{0,0}$ is also the  Laplace-Beltrami
operator for the Bergman metric on $\mathcal{U}^{n+1}$.

Then, as one of the main theorems in \cite{LuYangpaper}, Lu and Yang established the following factorization theorem for the  operators on the complex hyperbolic space. This factorization theorem, together with the Helgason-Fourier analysis on hyperbolic spaces  \cite{he, he2}, has played an important role   in proving the Hardy-Sobolev-Maz'ya and Hardy-Adams inequalities on complex hyperbolic spaces. (See \cite{LuYangpaper}.)

\begin{theorem}\label{th1.6a}
Let $a\in\mathbb{R}$ and $k\in\mathbb{N}\setminus\{0\}$.  In terms of  the Siegel domain model, we have,  for $u\in C^{\infty}(\mathcal{U}^{n+1})$,
\begin{equation}\label{a1.5}
\begin{split}
&\prod^{k}_{j=1}\left[q\partial_{qq}+a\partial_{q}+q T^{2}-\mathcal{L}_{0} -i(k+1-2j)T\right](q^{\frac{k-n-1-a}{2}} u) \\
=&4^{-k}q^{-\frac{k+n+1+a}{2}}\prod^{k}_{j=1}\left[\Delta_{\mathbb{B}}+(n+1)^{2}-(a-k+2j-2)^{2}\right]u.
\end{split}
 \end{equation}
In terms of  the ball  model, we have,  for $f\in C^{\infty}(\mathbb{B}_{\mathbb{C}}^{n+1})$,
  \begin{equation}\label{a1.6}
\begin{split}
&\prod^{k}_{j=1}\left[\Delta'_{\frac{-a-n}{2},\frac{-a-n}{2}}+\frac{(k+1-2j)^{2}}{4}-
\frac{k+1-2j}{2}(R-\bar{R})\right][(1-|z|^{2})^{\frac{k-n-1-a}{2}} f] \\
=&4^{-k}(1-|z|^{2})^{-\frac{k+n+1+a}{2}}\prod^{k}_{j=1}\left[\Delta_{\mathbb{B}}+(n+1)^{2}-(a-k+2j-2)^{2}\right]f.
\end{split}
 \end{equation}
 \end{theorem}

We remark that the operators on the left side of (\ref{a1.5}) and (\ref{a1.6})  are  closely related to Geller's operator, as well as the CR invariant differential operators on the Heisenberg group and CR sphere, respectively. The CR invariant differential operators on the Heisenberg group and CR sphere  have also played an important role in the work of Jerison-Lee on the CR Yamabe problem \cite{JerisonLee1, JerisonLee2}.
 This factorization theorem (Theorem \ref{th1.6a}) was recently used in \cite{Gunhee} to drive the second order Sobolev trace inequality in the CR setting.

\smallskip

As the application of Theorem \ref{th1.6a} and together with the Poincar\'e-Sobolev inequalities in \cite{LuYangpaper}, we have established in \cite{LuYangpaper} the following Hardy-Sobolev-Maz'ya inequalities on the complex hyperbolic space.

\begin{theorem}
\label{th1.8a}
Let $a\in\mathbb{R}$,  $1\leq k<n+1$ and $2<p\leq\frac{2n+1}{n+1-k}$ and $\mathbb{H}^{n}$ is the Heisenberg group. In terms of  the Siegal domain model, there exists a positive constant $C$ such that   for each $u\in C^{\infty}_{0}(\mathcal{U}^{n+1})$ we have
\begin{equation*}
  \begin{split}
 & \int_{\mathbb{H}^{n}}\int^{\infty}_{0}u\prod^{k}_{j=1}\left[-q\partial_{qq}-a\partial_{q}-q T^{2}+\mathcal{L}_{0} +i(k+1-2j)T\right]u\frac{dzdtdq}{q^{1-a}}\\
  &-\prod^{k}_{j=1}\frac{(a-k+2j-2)^{2}}{4}\int_{\mathbb{H}^{n}}\int^{\infty}_{0}\frac{u^{2}}{q^{k+1-a}}dzdtdq\\
 \geq&C\left(\int_{\mathbb{H}^{n}}\int^{\infty}_{0}|u|^{p}q^{\gamma}dzdtdq\right)^{\frac{2}{p}},
  \end{split}
\end{equation*}
where $\gamma=\frac{(n+1-k+a)p}{2}-n-2$. In terms of  the ball  model, we have  for $f\in C_{0}^{\infty}(\mathbb{B}_{\mathbb{C}}^{n+1})$,
\begin{equation*}
  \begin{split}
 &\int_{\mathbb{B}_{\mathbb{C}}^{n+1}}f \prod^{k}_{j=1}\left[\Delta'_{\frac{-a-n}{2},\frac{-a-n}{2}}+\frac{(k+1-2j)^{2}}{4}-
\frac{k+1-2j}{2}(R-\bar{R})\right]f\frac{dz}{(1-|z|^{2})^{1-a}}\\
  &-\prod^{k}_{j=1}\frac{(a-k+2j-2)^{2}}{4}\int_{\mathbb{B}_{\mathbb{C}}^{n+1}}\frac{f^{2}}{(1-|z|^{2})^{k+1-a}}dz\\
 \geq&C\left(\int_{\mathbb{B}_{\mathbb{C}}^{n+1}}|f|^{p}(1-|z|^{2})^{\gamma}dz\right)^{\frac{2}{p}}.
  \end{split}
\end{equation*}
\end{theorem}

Geller type operators on quaternionic and octonionic hyperbolic spaces have been introduced by Flynn, Lu and Yang in \cite{FlynnLuYangpaper} and a factorization theorem has been established with applications to prove the Hardy-Sobolev-Maz'ya and Hardy-Adams inequalities on quaternionic and octonionic hyperbolic spaces by using the Helgason-Fourier analysis on symmetric spaces \cite{he, he2}. We note that the authors established earlier the  Hardy-Sobolev-Maz'ya and Hardy-Adams inequalities on real hyperbolic spaces
(see \cite{LuYang3} and \cite{LiLuy2, LiLuy3, ly4, ly2}).

\medskip

 Inspired by the higher order factorization theorem (Theorem \ref{th1.6a}) and the higher order Hardy-Sobolev-Maz'ya inequalities (Theorem \ref{th1.8a}) and the interplay between the Geller's operators and the Laplace-Beltrami operators in the Siegel domain and the complex ball, we will establish  in this paper the higher order Sobolev trace inequalities in the Siegel domain and complex ball as one of our main results.

\smallskip
Motivated by the factorization theorem (Theorem \ref{th1.6a}), we introduce the following operator and
adopt the notation in the remaining part of this paper:
 \begin{align}\label{1.17a}
L_{2k}=&(-1)^{k}\prod^{k}_{j=1}\left[\partial_{\rho\rho}+\frac{1-2[\gamma]}{\rho}\partial_{\rho}+\rho^{2} T^{2}+\Delta_{b} -i(k+1-2j)T\right],\;\;\rho=(2q)^{\frac{1}{2}},
 \end{align}
where we have used the notations: for $\g \in (0, \infty)\setminus\N$,  let $\lfloor \g \rfloor$ be the integer part of $\g$, let
$
k = \lfloor \g \rfloor + 1
$
 and let $[\g] = \g - \lfloor \g \rfloor$ be the fractional part. By Theorem \ref{th1.6a}, we have
 \begin{align}\label{1.17b}
L_{2k} = (-1)^{k}\rho^{-(n+1) + \g - 2k} \circ\prod_{j=0}^{k-1}\left( \Delta_{\B} +(n+1)^{2}- (\g - 2j)^{2} \right)\circ\rho^{n+1-\g}.
 \end{align}

\smallskip

The operator $L_{2k}$ is naturally associated with the CR structure of the CR sphere on $\mathbb{H}^{n}$.
Indeed, using the factorization theorem of Lu and Yang \cite{LuYangpaper} (i.e., Theorem \ref{th1.6a} above),
 one finds that $L_{2k}$ is exactly the operator considered in \cite{LuYangpaper}, namely, a product of Geller's operators. As we shall see below,
 this operator $L_{2k}$ will also play a vital role in our defining of the conformally covariant boundary operators and establishing the higher order CR Sobolev trace inequalities on
 the Siegel domain $\mathcal{U}^{n+1} \cong \H^{n} \times (0,\oo) \subset \C^{n+1}$  and the  complex ball $\B_{\C}^{n+1} \subset \C^{n+1}$ in this paper.

\smallskip

It is known that  the CR invariant
sub-Laplacian of Jerison and Lee \cite{JerisonLee1,JerisonLee2}  in  CR geometry  plays a role analogous to that of the
conformal Laplacian in Riemaniann  geometry. In the case of CR sphere  $\mathbb{S}^{2n+1}$, the  CR invariant
sub-Laplacian is defined as $\mathcal{L}'_{0}+\frac{n^{2}}{4}$ (see \cite{MR0309156}), where $\mathcal{L}'_{0}$ is the Folland-Stein operator (see \eqref{eq:folland-stein-operator})  on $\mathbb{S}^{2n+1}$.
The CR covariant relationship between the sub-Laplacian $\mathcal{L}_{0} = -\frac{1}{2}\Delta_{b}$ on the Heisenberg group $\mathbb{H}^{n}$ and CR invariant sub-Laplacian $\mathcal{L}'_{0}+\frac{n^{2}}{4}$ on $\mathbb{S}^{2n+1}$ may be expressed as follows.
Letting $\partial\mathcal{C}$ be the restriction of Cayley transform $\mathcal{C}$ defined in \eqref{1.24} to the boundary $\mathbb{S}^{2n+1}$, letting $J_{\p C}$ be the Jacobian of $\p C$ and letting $Q=2n+2$ be the homogeneous dimension of $\mathbb{H}^{n}$ and $\mathbb{S}^{2n+1}$, there holds
\[
\left(\mathcal{L}'_{0}+\frac{n^{2}}{4}\right)\left(|J_{\partial\mathcal{C}}|^{\frac{Q-2}{2Q}}(F\circ \partial\mathcal{C})\right)=
|J_{\partial\mathcal{C}}|^{\frac{Q-2}{2Q}}\mathcal{L}_{0}F,\;\;F\in C^{\infty}(\mathbb{H}^{n}).
\]

The CR invariant differential operators for high order on the Heisenberg group are the product of  the Folland-Stein
operators. In fact, if we denote the  Folland-Stein
operators $\mathcal{L}_{\alpha}$ (see \cite{MR367477}) by
$$\mathcal{L}_{\alpha}=\mathcal{L}_{0}+i\alpha T,$$
then the  CR invariant differential operator of order $2k$ is, up to a constant,
\begin{equation}\label{}
  N_{2k}=\prod^{k}_{j=1}\mathcal{L}_{k+1-2j}.
\end{equation}

Similarly, there is an analogous operator  on the CR sphere. Let
$$\mathcal{T}=\frac{i}{2}(R-\overline{R})$$
be the transversal direction. The  CR invariant differential operator of order $2k$ on $\mathbb{S}^{2n+1}$ is given by, up to a constant,
$$N'_{2k}=\prod^{k}_{j=0}\left(\mathcal{L}'_{0}+\frac{n^{2}}{4}-\frac{(k+1-2j)^{2}}{4}-(k+1-2j)i\mathcal{T}\right).$$
The relationship between $N_{2k}$ and the $N'_{2k}$
is (see \cite{Graham1})
\[
N'_{2k}\left(|J_{\partial\mathcal{C}}|^{\frac{Q-2k}{2Q}}(F\circ \partial\mathcal{C})\right)=
|J_{\partial\mathcal{C}}|^{\frac{Q-2k}{2Q}}N_{2k}F,\;\;F\in C^{\infty}(\mathbb{H}^{n}).
\]
For the conformally invariant sharp  Sobolev inequality on the Heisenberg group and CR sphere, we refer to Jerison-Lee \cite{JerisonLee1, JerisonLee2} and Frank-Lieb \cite{MR2925386}.

\smallskip

We now state our results, beginning with the results for the Siegel domain and Heisenberg group.
To begin, we fix notations.
We will let $\Delta_{\B}$ denote the Laplace-Beltrami operator \eqref{eq:laplace-beltrami-operator-siegel-domain} on $\mathcal{U}^{n+1}$ equipped with the complex hyperbolic metric $g^{+}$ given in \eqref{eq:kahler-metric-on-siegel-domain}.
Note that our normalization of $\Delta_{\B}$ is different than that of Frank et al in \cite{Frank2}.
We fix the parameters
\begin{align*}
  \g & \in (0,\infty) \setminus \N\\
  \lfloor \g \rfloor &= \text{ integer part of }\g\\
  k &= \lfloor \g \rfloor + 1\\
  [\g] &= \g - \lfloor \g \rfloor = \text{ fractional part of } \g.
\end{align*}
 For notational simplicity, we will usually use the following notation:
\begin{align}\label{1.13}
\tilde\Delta_{\mathbb{B}} = \Delta_{\mathbb{B}} + (n+1)^{2}.
\end{align}
We introduce the following weighted operators:
\begin{equation}\label{1.13a}
  D_{s} = \Delta_{\mathbb{B}} + 4s(n+1-s), \qquad L_{2k}^{+} =(-1)^{k} \prod_{j=0}^{k-1}D_{s-j},\qquad s=\frac{n+1+\gamma}{2},
\end{equation}
By (\ref{1.17b}), we have immediately that
\begin{equation}\label{1.13b}
L_{2k}= \rho^{-(n+1+2k-\gamma) } \circ L_{2k}^{+} \circ \rho^{n+1-\gamma},\;\;\;\;\;\; k=\lfloor \g \rfloor + 1.
 \end{equation}
Such operators have deep connections with the representation theory associated with rank one symmetric spaces, as was further explored by the authors in \cite{FlynnLuYangpaper}.

\smallskip

\smallskip

As mentioned above, one of the main contributions of our paper is to introduce an appropriate family of conformally covariant boundary operators associated to the Siegel domain or Bergman ball and the weighted operators $L_{2k},L_{2k}^{+}$.
The boundary operators we introduce is a family of operators $B_{\a}^{2\g}$ acting on the function space $\mcC^{2\g}(\mcU^{n+1})$ consisting of those $u \in C^{\oo}(\overline{\mcU^{n+1}})$ which have suitable Taylor expansions in terms of $\rho^{2k}$ and $\rho^{2k+2[\g]}$.
 For simplicity, we let
 \begin{align}\label{1.15a}
 \mathcal{C}^{2\g}(\mcU^{n+1}) = C_{\operatorname{even}}^{\oo} ( \overline{\mcU^{n+1}}) + \rho^{2 [\g]} C_{\operatorname{even}}^{\oo}( \overline{\mcU^{n+1}}),
 \end{align}
where $f \in C_{\operatorname{even}}^{\oo}(\overline{\mcU^{n+1}})$ indicates that $f$ has a Taylor expansion in terms of even powers (including 0) of $\rho$; i.e.,
\[
  f(z,t,\rho) = \sum\limits_{j=0}^{\oo} f_{j}(z,t)\rho^{2j},
\]
where $f_{j}(z,t)$ are smooth functions depending only on $(z,t) \in \p \mcU^{n+1}$.
For $f\in \mathcal{C}^{2\gamma}(\mathcal{U}^{n+1})$ and $j\in \mathbb{N}$ in the indicated ranges below we define the boundary operators associated with $(\mathcal{U}^{n+1},\rho^{2}g^{+})$ as follows:
\begin{itemize}
\item $ B_{ 0}^{2\g}(f)=f|_{\rho=0}$;
  \item  for $1\leq j\leq \lfloor\gamma/2\rfloor$, define
  \begin{align}\label{1.18}
 B_{2 j}^{2\g}(f)=&\frac{1}{ b_{ 2j }}\rho^{-(n+1) + \g - 2j}   \prod_{\ell=0}^{j-1}D_{s-l} \prod_{\ell=\lfloor\gamma\rfloor-j+1}^{\lfloor\gamma\rfloor}D_{s-l} (\rho^{n+1-\g } f)|_{\rho=0};
  \end{align}
  \item for $0\leq j\leq \lfloor\gamma\rfloor-\lfloor\gamma/2\rfloor-1$, define
  \begin{align}\label{1.19}
   B_{ 2j + 2[\g]}^{2\g }(f)=&\frac{1}{ b_{ 2j + 2[\g]}}\rho^{-(n+1) + \g - 2j  - 2[\g]} \prod_{\ell=0}^{j}D_{s-l} \prod_{\ell=\lfloor\gamma\rfloor-j+1}^{\lfloor\gamma\rfloor}D_{s-l} ( \rho^{n+1-\g}f)  |_{\rho=0}
  \end{align}
\end{itemize}
where
\begin{align*}
 b_{2j} = &4^{2j}j!\frac{\Gamma(\gamma+1-j)\Gamma(\lfloor \g \rfloor+1-j)\Gamma(j+1-[\gamma])}{\Gamma(\gamma+1-2j)\Gamma(\lfloor \g \rfloor+1-2j)\Gamma(1-[\gamma])},\;\;\;\;\;\;\;\;\;\;\;\;\;\;\;\;\;\;\;\;\;\;1\leq j\leq \lfloor\gamma/2\rfloor;\\
  b_{2j+2[\g]} =  &-4^{2j+1}j!\frac{\Gamma(j+1+[\gamma])\Gamma(\lfloor\gamma\rfloor+1-j)\Gamma(\lfloor \g \rfloor+1-j-[\gamma])}{\Gamma([\gamma])\Gamma(\lfloor\gamma\rfloor-2j)\Gamma(\lfloor \g \rfloor+1-2j-[\gamma])},\;\;0\leq j\leq \lfloor\gamma\rfloor-\lfloor\gamma/2\rfloor-1
\end{align*}
 are chosen   such that (see Lemma \ref{lm1.2})
 \begin{align*}
 B_{2 j}^{2\g}(\rho^{2j})= B_{ 2j + 2[\g]}^{2\g }(\rho^{2j + 2[\g]})=1.
 \end{align*}

However,  $b_{2j}=0$ when $\lfloor\gamma/2\rfloor+1\leq j\leq\lfloor\gamma\rfloor$  and   $ b_{2j+2[\g]}=0$ when $\lfloor\gamma\rfloor-\lfloor\gamma/2\rfloor\leq j\leq\lfloor\gamma\rfloor$.  In such cases, we
 define the boundary operators   associated with $(\mathcal{U}^{n+1},\rho^{2}g^{+})$ in terms of the solution $L_{2k} f = 0$.
 This is done in Section \ref{Section4}.

\smallskip

Our main results for the Siegel domain and Heisenberg group may now be stated.
To see that the boundary operators are natural, we will prove they satisfy the following four properties, the latter two being the higher order extension results and higher order trace inequalities, respectively.
The first  is the conformal covariance property.
Given another defining function $\widehat{\rho}=e^{\tau}\rho$ and if we let $\widehat{B}_{2j}^{2\gamma}$ and $\widehat{B}_{2j+2[\gamma]}$ be the boundary operators associated with $(\mathcal{U}^{n+1}, \widehat{\rho}^{2}g_{+})$, then we have the following theorem.
\begin{theorem}\label{thm:boundary-operator-conformal-covariance}
    Let  $V\in \mathcal{C}^{2\gamma}(\mathcal{U}^{n+1})$ and $0\leq j\leq\lfloor\gamma\rfloor$.
  It holds that
  \begin{align}
    \widehat{B}_{2j}^{2\gamma}(V)=&e^{(-(n+1) + \g  - 2j)\tau|_{\mathbb{H}^{n}}}B_{2j}^{2\gamma}(e^{(n+1-\gamma)\tau}V);\\
    \widehat{B}_{2j+2[\gamma]}^{2\gamma}(V)=&e^{(-(n+1) + \g  - 2j  - 2[\g])\tau|_{\mathbb{H}^{n}}}B_{2j+2[\gamma]}^{2\gamma}(e^{(n+1-\gamma)\tau}V).
  \end{align}
  where $\tau|_{\mathbb{H}^{n}}$ is the restriction of $\tau$ on $\mathbb{H}^{n}$.
\end{theorem}

The second is that the associated Dirichlet form
\begin{align*}
    \mathcal{Q}_{2\gamma}(U,V):=&\int_{\mathcal{U}^{n+1}}U L_{2k} V \cdot\rho^{1-2[\gamma]}dzdtd\rho \\
    &- \sum_{j=0}^{\lfloor \g/2 \rfloor} \sigma_{j,\g}  \int_{\p\mcU^{n+1}} B_{2j}^{2\g}(U)  B_{2\g-2j}^{2\g}(V)  dzdt - \sum_{j=\lfloor \g/2 \rfloor + 1}^{\lfloor \g \rfloor} \sigma_{j,\g} \int_{\p\mcU^{n+1}}   B_{2\g-2j}^{2\g}(U)  B_{2j}^{2\g}(V)  dzdt
\end{align*}
is symmetric for $U,V\in \mathcal{C}^{2\gamma}(\mathcal{U}^{n+1})\cap\dot{H}^{k,[\gamma]}(\mathcal{U}^{n+1})$, where
\begin{align}\label{1.22}
    \sigma_{j,\g} &=
    \begin{cases}
       2^{2\lfloor\g\rfloor+1}j!(\lfloor\g\rfloor-j)!\frac{\Gamma(\gamma+1-j)\Gamma(j+1-[\g])}{\Gamma(\gamma-2j)\Gamma(2j+1-\g)},& j = 0,\ldots, \lfloor \g/2 \rfloor\\
     -2^{2\lfloor\g\rfloor+1}j!(\lfloor\g\rfloor-j)!\frac{\Gamma(\gamma+1-j)\Gamma(j+1-[\g])}{\Gamma(\gamma-2j)\Gamma(2j+1-\g)}, & j = \lfloor \g /2 \rfloor + 1,\ldots, \lfloor \g \rfloor.\\
    \end{cases}
  \end{align}
This is recorded in the following theorem.
\begin{theorem}\label{thm:dirichlet-form-symmetry}
  Let $U,V\in \mathcal{C}^{2\gamma}(\mathcal{U}^{n+1})\cap\dot{H}^{k,[\gamma]}(\mathcal{U}^{n+1})$.
It holds that
\begin{align*}
\mathcal{Q}_{2\gamma}(U,V)=\mathcal{Q}_{2\gamma}(V,U).
\end{align*}
\end{theorem}

Third, the boundary operators recover the CR covariant operators $P_{\g}$ on $\H^{n}$ in the following way.
\begin{theorem}\label{thm:boundary-to-fractional-operator}
  Given $V \in \CHspace$ satisfying $L_{2k} V = 0$, there holds
\begin{align}
B^{2\gamma}_{2\gamma-2j}(V)=&\frac{2^{\gamma-2j}}{c_{2j-\gamma}}P_{\gamma-2j}B^{2\gamma}_{2j}(V),\;\;0\leq j\leq\lfloor\gamma/2\rfloor;\\
B^{2\gamma}_{2\lfloor\gamma\rfloor-2j}(V)=&\frac{2^{\lfloor\gamma\rfloor-2j-[\gamma]}}{c_{2j+[\gamma]-\lfloor\gamma\rfloor}}
P_{\lfloor\gamma\rfloor-[\gamma]-2j}B^{2\gamma}_{2j+2[\gamma]}(V),\;\;
0\leq j\leq\lfloor\gamma\rfloor-\lfloor\gamma/2\rfloor-1,
\end{align}
where $c_{\gamma}$ is given in (\ref{1.8}).
\end{theorem}

Lastly, by using these boundary operators, we have the following sharp CR Sobolev trace inequality.
\begin{theorem}\label{th1.5}
  Let  $\gamma\in (0,n+1)\setminus\mathbb{N}$ and $U\in \mathcal{C}^{2\gamma}(\mathcal{U}^{n+1})\cap\dot{H}^{k,[\gamma]}(\mathcal{U}^{n+1})$ and set $\mathcal{E}_{2\gamma}(U)=\mathcal{Q}_{2\gamma}(U,U)$.
  It holds that
  \begin{align*}
    \mathcal{E}_{2\g}(U) \geq& \sum_{j=0}^{\lfloor \g/2 \rfloor} \varsigma_{j,\g} \int_{\p\mcU^{n+1}}   B_{2j}^{2\g}(U) P_{\g-2j}B_{2j}^{2\g} (U)    dzdt \\
    &+ \sum_{j=\lfloor \g /2 \rfloor + 1}^{\lfloor \g \rfloor} \varsigma_{j,\g} \int_{\p\mcU^{n+1}}   B_{2\g - 2j}^{2\g}(U) P_{2j - \g}B_{2\g-2j}^{2\g} (U)    dzdt.
  \end{align*}
  where
\begin{align}\label{1.27}
    \varsigma_{j,\g} &=
    \begin{cases}
       2^{4j-2[\g]+1}j!(\lfloor\g\rfloor-j)!\frac{\Gamma(\gamma+1-j)\Gamma(j+1-[\gamma])}{\Gamma(\gamma+1-2j)\Gamma(\gamma-2j)}, & j = 0,\ldots, \lfloor \g /2 \rfloor\\
      2^{4\gamma-4j-2[\g]+1}j!(\lfloor\g\rfloor-j)!\frac{\Gamma(\gamma+1-j)\Gamma(j+1-[\gamma])}
      {\Gamma(2j+1-\gamma)\Gamma(2j-\gamma)}, & j = \lfloor \g/2 \rfloor + 1, \ldots, \lfloor \g \rfloor.
    \end{cases}
  \end{align}
  are positive constants.
  Moreover, equality is attained iff $L_{2k} U = 0$.
\end{theorem}

\begin{remark}\label{rem:energy-minimization}
  The equality case in the CR Sobolev trace inequality has the following immediate interpretation.
  Given boundary data
  \begin{align*}
    f^{(2j)} \in C^{\oo}(\p\mathcal{U}^{n+1}) \cap S^{\g-2j,2}(\p\mathcal{U}^{n+1}) \qquad   &\text{ for } \qquad j = 0, \ldots \lfloor \g/2 \rfloor\\
    \phi^{(2j)}  \in C^{\oo}(\p\mathcal{U}^{n+1}) \cap S^{\lfloor \g \rfloor - [\g] -2j,2}(\p\mathcal{U}^{n+1}) \qquad & \text{ for } \qquad j = \lfloor \g/2 \rfloor + 1, \ldots, \lfloor \g \rfloor,
  \end{align*}
  then over all $V\in \mathcal{C}^{2\gamma}(\mathcal{U}^{n+1})\cap\dot{H}^{k,[\gamma]}(\mathcal{U}^{n+1})$ with boundary values
  \begin{align*}
    B_{2j}^{2\g}(V)  & = f^{(2j)},  \quad j = 0, \ldots \lfloor \g/2 \rfloor\\
    B_{2\g-2j}^{2\g} (V) &= \phi^{(2j)},  \quad j = \lfloor \g/2 \rfloor + 1, \ldots, \lfloor \g \rfloor,
  \end{align*}
  the energy $\mathcal{E}_{2\g}$ is minimized exactly at the unique solution to the Dirichlet problem
  \begin{equation*}
    \begin{cases}
      L_{2k} V = 0, & \text{in }\mcU^{n+1}\\
      B_{2j}^{2\g}(V) = f^{(2j)},& 0 \leq j \leq \lfloor \g/2 \rfloor\\
      B_{2\g-2j}^{2\g}(V) = \phi^{(2j)}, & \lfloor \g/2 \rfloor + 1 \leq j\leq \lfloor \g \rfloor.
    \end{cases}.
  \end{equation*}

\end{remark}

We now state the corresponding results on the ball model of complex hyperbolic space. We should mention that in the case of real ball in the Euclidean space, Sobolev trace inequalities were established in \cite{Escobar,Beckner1,AC,NN,YangQ}.
For convenience, we set
\[\mathbb{B}_{\mathbb{C}}^{n+1}=\{w=(w_{1},\cdots,w_{n+1})\in \mathbb{C}^{n+1}: |w|<1\}.\]
The ball model of complex hyperbolic space is given by the unit ball
$\mathbb{B}_{\mathbb{C}}^{n+1}$
equipped with the K\"ahler metric $-\frac{i}{2}\partial\bar{\partial}\log (1-|w|^{2}).$

\smallskip
Let $\Delta_{\B}$ now indicate the Laplace-Beltrami operator in the complex ball $\B_{\C}^{n+1}$.
We then introduce the following weighted operators on $\mathbb{B}_{\mathbb{C}}^{n+1}$:
\begin{align*}
\tilde\Delta_{\mathbb{B}}& = \Delta_{\mathbb{B}} + (n+1)^{2},\\
  D_{s,\mathbb{B}} &= \Delta_{\mathbb{B}} + 4s(n+1-s), \qquad L_{2k,\mathbb{B}}^{+} = (-1)^{k}\prod_{j=0}^{k-1}D_{s-j,\mathbb{B}},\qquad s=\frac{n+1+\gamma}{2},\\
L_{2k,\mathbb{B}}&= (1-|w|^{2})^{-\frac{n+2+k-[\gamma]}{2} } \circ L_{2k,\mathbb{B}}^{+} \circ (1-|w|^{2})^{\frac{n+2-[\gamma]-k}{2}}.
\end{align*}
The boundary operators $B_{\a,\B}^{2\g}$ for $\B_{\C}^{n+1}$ and the space $\mathcal{C}^{2\g}(\mathbb{B}_{\C}^{n+1})$ are defined analogously to the boundary operators on $\mcU^{n+1}$ and $\mathcal{C}^{2\g}(\mcU^{n+1})$, respectively (see Section \ref{sec:scattering-problem-solution} for their precise definitions).

\smallskip

We remark that the two models,  $\mathcal{U}^{n+1}$ and $\B_{\C}^{n+1}$,  are related by the Cayley transform.
This also induces the CR equivalence of the Heisenberg group $\mathbb{H}^{n}$ with the CR sphere $\mathbb{S}^{2n+1}$.
We recall that the Cayley transform $\mathcal{C}: \mathcal{U}^{n+1} \rightarrow \mathbb{B}_{\mathbb{C}}^{n+1}$ (see \cite{Stein2,gra1}), which is defined as the mapping
\begin{align}\label{1.24}
  \mathcal{C}(z)= \left(\frac{2iz_{1}}{i+z_{n+1}}, \cdots, \frac{2iz_{n}}{i+z_{n+1}},\frac{i-z_{n+1}}{i+z_{n+1}}\right), \quad z \in \mathcal{U}^{n+1},
\end{align}
is an isometry between the two models of complex hyperbolic spaces.
The inverse of $\mathcal{C}$ is given by
\begin{equation}\label{1.25}
\mathcal{C}^{-1}: \mathbb{B}_{\mathbb{C}}^{n+1}\rightarrow\mathcal{U}^{n+1},
w\mapsto \left(\frac{w_{1}}{1+w_{n+1}}, \cdots, \frac{w_{n}}{1+w_{n+1}},i\frac{1-w_{n+1}}{1+w_{n+1}}\right).
\end{equation}
Letting $\partial\mathcal{C}$ be the restriction of the Cayley transform $\mathcal{C}$  to the boundary, then $\partial\mathcal{C}$ is nothing but the Cayley transform on Heisenberg group $\mathbb{H}^{n}$ (see e.g. \cite{MR2925386,Branson2}).

By using the Cayley  transform $\mathcal{C}$, the boundary operators $B_{\a,\B}^{2\g}$ for $\B^{n+1}_{\C}$ and $B_{\a}^{2\g}$ for $\mcU^{n+1}$ are related in the following way.

\begin{theorem}\label{th1.9} Let $0\leq j\leq \lfloor\gamma\rfloor$.
It holds that, for $g\in \mathcal{C}^{2\gamma}(\mathbb{B}_{\mathbb{C}}^{n+1})$,
\begin{align*}
B_{2 j, \mathbb{B}}^{2\g}(g)=&|2J_{\partial\mathcal{C}}|^{\frac{-(n+1) + \g  - 2j}{2n+2}}
B_{2 j}^{2\g}\left(|4J_{\mathcal{C}}|^{\frac{n+1-\gamma}{2(n+2)}} g(\mathcal{C}(z))\right);\\
B_{2 j+2[\gamma], \mathbb{B}}^{2\g}(g)=&|2J_{\partial\mathcal{C}}|^{\frac{-(n+1) + \g  - 2j-2[\gamma]}{2n+2}}
B_{2j+2[\gamma]}^{2\g}\left(|4J_{\mathcal{C}}|^{\frac{n+1-\gamma}{2(n+2)}} g(\mathcal{C}(z))\right),
\end{align*}
where $J_{\mathcal{C}}$  and $J_{\partial\mathcal{C}}$ are the  Jacobian determinant of $\mathcal{C}$ and $\partial\mathcal{C}$, respectively.
\end{theorem}

For $0 < \g < n+1$, let $P_{\g}^{\S^{2n+1}}$ denote the intertwining operator on $\S^{2n+1}$ of order $\g/2$ (see \cite{Branson1,Branson2,Graham1}).
That is, if $Y_{j,k}$ is a spherical harmonic of bidegree $(j,k)$, then
\begin{align}\label{eq:1.22}
P_{\g}^{\S^{2n+1}} Y_{j,k} = 2^{\gamma}\frac{\Gamma\left( \frac{n+1+\g}{2} + j \right)}{\Gamma\left( \frac{n+1-\g}{2} + j \right)}\frac{\Gamma\left( \frac{n+1+\g}{2} + k \right)}{\Gamma\left( \frac{n+1-\g}{2} + k \right)} Y_{j,k}.
\end{align}
Moreover,  $P_{\g}^{\S^{2n+1}}$ and $P_{\g}$ are CR covariantly related in the sense that (see \cite{Branson2})
\begin{align}\label{1.29}
  P_{\g}^{\S^{2n+1}}(h)\circ\mathcal{C} = |2J_{\p C}|^{- \frac{n+1 + \g}{2n+2}}\circ P_{\g}\left( |2J_{\p C}|^{\frac{n+1-\g}{2n+2}}h\circ \mathcal{C}\right),\; h\in C^{\infty}(\mathbb{S}^{2n+1}),\; 0<\gamma<n+1.
\end{align}

\smallskip

 We  obtain also that the boundary operators on $\B_{\C}^{n+1}$ recover the fractional operators $P_{\g}^{\S^{2n+1}}$ on the CR sphere.
\begin{theorem}\label{th1.10}
  Given $u \in \CHspaceball$ satisfying $L_{2k,\B} u = 0$, there holds
\begin{align}\label{1.30}
  B^{2\gamma}_{2\gamma-2j,\B}(u)=&\frac{2^{\gamma-2j}}{c_{2j-\gamma}} P_{\gamma-2j}^{\S^{2n+1}}B^{2\gamma}_{2j,\B}(u),\;\;0\leq j\leq\lfloor\gamma/2\rfloor;\\
  \label{1.31}
  B^{2\gamma}_{2\lfloor\gamma\rfloor-2j,\B}(u)=&\frac{2^{\lfloor\gamma\rfloor-2j-[\gamma]}}{c_{2j+[\gamma]-\lfloor\gamma\rfloor}}
  P_{\lfloor\gamma\rfloor-[\gamma]-2j}^{\S^{2n+1}}B^{2\gamma}_{2j+2[\gamma],\B}(u),\;\;
0\leq j\leq\lfloor\gamma\rfloor-\lfloor\gamma/2\rfloor-1.
\end{align}
\end{theorem}

Next, for $u,v\in \mathcal{C}^{2\gamma}(\mathbb{B}_{\mathbb{C}}^{n+1})$, set
\begin{align*}
  \mathcal{Q}_{2\gamma,\mathbb{B}}(u,v):=&2^{1-\gamma}\int_{\mathbb{B}_{\mathbb{C}}^{n+1}}u L_{2k,\mathbb{B}} v \cdot(1-|z|^{2})^{-[\gamma]}dz\\
  &-\sum_{j=0}^{\lfloor\gamma/2\rfloor} \sigma_{j,\g} \int_{\mathbb{S}^{2n+1}}B_{2j,\mathbb{B}}^{2\g}(u)B^{2\gamma}_{2\gamma-2j,\mathbb{B}}(v)d\sigma -\sum_{j=\lfloor \g /2 \rfloor + 1 }^{\lfloor \g \rfloor} \sigma_{j,\g}
\int_{\mathbb{S}^{2n+1}}B_{2\g-2j,\mathbb{B}}^{2\g}(u)B^{2\gamma}_{2j}(v)d\sigma
\end{align*}
and
\begin{align*}
\mathcal{E}_{2\gamma,\mathbb{B}}(u)=\mathcal{Q}_{2\gamma,\mathbb{B}}(u,u).
\end{align*}
We have also the following higher order CR Sobolev trace inequalities for the CR sphere $\S^{2n+1}$.
\begin{theorem}\label{th1.11}
Let  $\gamma\in (0,n+1)\setminus\mathbb{N}$ and $u\in \mathcal{C}^{2\gamma}(\mathbb{B}_{\mathbb{C}}^{n+1})$.
It holds that
  \begin{align*}
    \mathcal{E}_{2\g,\B}(u) \geq& \sum_{j=0}^{\lfloor \g/2 \rfloor} \varsigma_{j,\g} \int_{\mathbb{S}^{2n+1}}   B_{2j,\B}^{2\g}(u) P_{\g-2j}^{\mathbb{S}^{2n+1}}B_{2j,\B}^{2\g} (u)  d\sigma  +\\
     &\sum_{j=\lfloor \g /2 \rfloor + 1}^{\lfloor \g \rfloor} \varsigma_{j,\g} \int_{\mathbb{S}^{2n+1}}   B_{2\g - 2j,\B}^{2\g}(u) P_{2j - \g}^{\mathbb{S}^{2n+1}}B_{2\g-2j,\B}^{2\g} (u)    d\sigma,
  \end{align*}
  where $\varsigma_{j,\g}(j=0,1,\cdots, \lfloor \g \rfloor)$  are given in  (\ref{1.27}).
  Moreover, equality is attained iff $L_{2k,\B} U = 0$.
\end{theorem}

Naturally, the equality case in Theorem \ref{th1.11} has a similar energy minimization interpretation as that given in Remark \ref{rem:energy-minimization}.

\medskip
As mentioned above, the last of the main results for this paper is determining an explicit solution for the scattering problem on the complex ball.
Namely, we obtain explicit expressions for the solution $u$ to the Poisson equation
\begin{equation*}
  \Delta_{\varphi} u- s(n+1-s)u=0 \quad  \text{ in } \quad \B_{\C}^{n+1},
\end{equation*}
where $\Delta_{\varphi} = \frac{1}{4} \Delta_{\B}$, where $u$ is the unique solution satisfying
\begin{align*}
    u&=\varphi^{n+1-s}F+\varphi^{s}G\\
    F|_{M}&=f
\end{align*}
and where $\g\in(0,\frac{n+1}{2})$, $s=\frac{n+1}{2}+\g$, $\varphi = 1-|w|^{2}$ and $f \in C^{\oo}(\mathbb{S}^{2n+1})$.
The result is recorded in the following theorem.
We point the reader to Section \ref{sec:preliminaries} for notation concerning the spherical harmonics $Y_{j,k}$ and hypergeometric functions $F(a,b,c;z)$.

\begin{theorem}\label{thm:scattering-solution}
  Let $\g\in(0,\frac{n+1}{2})$, $s=\frac{n+1}{2}+\g$, and let $\varphi=1-|w|^{2}$, and $\Delta_{\varphi} = \frac{1}{4} \Delta_{\B}$.
  Then the solution to the scattering problem
  \begin{equation}
    \begin{cases}
      \Delta_{\varphi}u-s(n+1-s)u=0 \quad \text{ in } \quad \B_{\C}^{n+1}\\
      u=\varphi^{n+1-s}F+\varphi^{s}G\\
      F|_{M}=f
    \end{cases}
    \label{eq:scattering-problem-on-complex-hyperbolic-space}
  \end{equation}
  is
  \begin{equation}
    u(w)=\frac{\Gamma^{2}(s)}{n!\Gamma(2\g)}\int_{\S^{2n+1}}\left( \frac{1-|w|^{2}}{|1-w\cdot\bar\xi|^{2}} \right)^{s}f(\xi)d\sigma.
    \label{eq:scattering-problem-solution-as-integral}
  \end{equation}
  Furthermore, if $f$ has the expansion in spherical harmonics $f=\sum\limits_{j,k=0}^{\oo}Y_{j,k}$, then
  \begin{equation}
    u(w)=\varphi^{n+1-s}\sum_{j,k=0}^{\oo}\varphi_{j,k}(r^{2})r^{j+k}Y_{j,k},
    \label{eq:scattering-problem-solution-as-sum}
  \end{equation}
  where $r=|w|$ and
  \[
    \varphi_{j,k}(r^{2})=\frac{\Gamma(j+s)\Gamma(k+s)}{\Gamma(j+k+n+1)\Gamma(2\g)}F(j+n+1-s,k+n+1-s,j+k+n+1;r^{2})
  \]
  with $\varphi_{j,k}(1)=1$.
\end{theorem}

Our approach to Theorem \ref{thm:scattering-solution} is based on the third author's work in \cite{YangQ} where they solved the scattering problem related to the real hyperbolic ball.

\medskip
We lastly provide an outline of the paper.
In Section \ref{sec:preliminaries}, we begin by recording various preliminaries that will be needed throughout the paper.
In Section \ref{sec:boundary-operators}, we introduce the boundary operators natural for the higher order CR Sobolev trace inequalities and the fractional CR covariant operators on the boundary.
In this section we also prove several important properties of these operators.
Then, in Section \ref{sec:dirichlet-problem-and-scattering}, we recall scattering theory on the complex hyperbolic space as well as study the Dirichlet problem related to the weighted operators $L_{2k}$ and our boundary operators. Section \ref{Section4} provides a proof of the conformal invariance of
the boundary operators on  $\mathcal{U}^{n+1}$, namely Theorem \ref{thm:boundary-operator-conformal-covariance}.
In Section \ref{sec:proofs-section}, we provide the proofs of our main results as well as prove an integral identity crucial for establishing the higher order CR Sobolev trace inequalities and the symmetry of $\mathcal{Q}_{2\g}$.
In Section \ref{sec:scattering-problem-solution} we give the proof of Theorems  \ref{th1.9}, \ref{th1.10} and \ref{th1.11}  on the complex hyperbolic ball.
Lastly, in Section \ref{sec:scattering-problem-solution2}, we find an explicit expression for solutions to the scattering problem on the complex ball.

\section{Preliminaries}
\label{sec:preliminaries}

\subsection{The Complex Ball}
By the complex ball we mean the unit ball $\mathbb{B}_{\C}^{n+1} \subset \mathbb{C}^{n+1}$ centered at the origin and equipped with the K\"ahler metric
\[
  g_{\B} = -\frac{i}{2}\p\bar\p \log (1-|w|^{2}).
\]
The metric $g_{\B}$ is often called the Bergman metric and the space $(\B_{\C}^{n+1},g_{\B})$ the Bergman ball or complex hyperbolic space $H_{\C}^{n+1}$.
As is the case for the real hyperbolic space, the complex hyperbolic space $H_{\C}^{n+1}$ has this Beltrami-Klein ball model as well as a halfspace mode, which is detailed in the next sections.
The Riemannian volume form associated to $g_{\B}$ is given by
\[
  dV = \frac{dz}{(1-|w|^{2})^{n+2}},
\]
where $dz$ is the usual Euclidean volume.
The associated Laplace-Beltrami operator is given by
\begin{equation}
  \Delta_{\B} =4 (1-|w|^{2})\sum_{j,k=1}^{n+1}(\d_{jk} - w_{j} \bar w_{k}) \frac{\p^{2}}{\p w_{j} \p \bar w_{k}},
  \label{eq:laplace-beltrami-operator-complex-ball}
\end{equation}
where
\[
\delta_{jk}=
\begin{cases}
  1, & \hbox{$j=k$;} \\
  0, & \hbox{$j\neq k$.}
\end{cases}
\]
As a real hypersurface in $\C^{n+1}$, the boundary $\mathbb{S}^{2n+1} = \p \B^{n+1}_{\C}$ is a CR manifold with standard induced contact form
\[
  \theta_{1} = i (\bar\p-\p)|w|^{2} = i \sum_{j=1}^{n+1}(w_{j} d \bar w_{j} - \bar w_{j}dw_{j}).
\]
Relative to this contact structure, the Folland-Stein operator \cite{MR367477} on $\mathbb{S}^{2n+1}$ is defined by
\begin{equation}
  \mathcal{L}_{0}' = -\frac{1}{2} \sum_{j<k}(M_{jk}\bar M_{jk} + \bar M_{jk}M_{jk})
  \label{eq:folland-stein-operator}
\end{equation}
where
\[
  M_{jk} = w_{j} \p_{\bar w_{k}} - \bar w_{k}\p_{w_{j}}.
\]

\subsection{Spherical Harmonics on $\mathbb{S}^{2n+1}$}

When $\B_{\C}^{n+1}$ is equipped with the Bergman metric, we have $\B_{\C}^{n+1}\cong{SU(n+1,1)/S(U(n+1)\times U(1))}$ as a symmetric spaces, and therefore the Bergman Laplacian
\[
  \Delta_{\varphi}=-(1-|w|^{2})\sum_{j,k=1}^{n+1}\left( \d_{jk}-w_{j}\bar{w}_{k} \right)\p_{w_{j}}\p_{\bar{w}_{k}}= \frac{1}{4}\Delta_{\mathbb{B}}
\]
is $U(n+1)$-invariant.
See \cite{MR0473215} for more on the Bergman metric and its Laplacian.
Therefore, the appropriate spherical harmonic decomposition of $L^{2}(\S^{2n+1})$ is into $U(n+1)$-irreducibles; i.e.,
\begin{equation}
  L^{2}(\S^{2n+1})=\bigoplus_{j,k=0}^{\oo}\mcH_{j,k},
  \label{eq:l2-cr-spherical-harmonic-decomposition}
\end{equation}
where $\mcH_{j,k}$ is $U(n+1)$-irreducible.
More precisely, $Y_{j,k}\in\mcH_{j,k}$ if and only if $Y_{j,k}=f(w,\bar{w})|_{\S^{2n+1}}$ for some harmonic (with respect to the Euclidean Laplacian) polynomial $f(w,\bar{w})$ which is homogeneous of degree $j$ in $w$ and of degree $k$ in $\bar{w}$.
Particularly, if $|w|=r$, then $f(w,\bar{w})=|w|^{j}|w|^{k}f(w/|w|,\bar{w}/|w|)=r^{j+k}Y_{j,k}$.
If $Y_{j,k} \in \mcH_{j,k}$, we say $Y_{j,k}$ is a spherical harmonic of bidegree $(j,k)$.
For more details, see \cite{MR0309156} and the references therein.

In \cite{YangQ}, to solve the scattering problem on the real hyperbolic ball, the Funk-Hecke formula was used to rewrite the integral of the fractional Poisson kernel against the boundary value $f$ in terms of spherical harmonics.
Thus, in order to follow this approach, we recall the CR Funk-Hecke formula of Frank and Lieb \cite{MR2925386}.
(Note that the constants appearing below differ from that appearing in \cite{MR2925386} due to our choice to use the normalized round measure on $\mathbb{S}^{2n+1}$ and hence any orthonormal basis respecting the decomposition \eqref{eq:l2-cr-spherical-harmonic-decomposition} differs by a constant than those considered in \cite{MR2925386}.)

\begin{proposition}[Frank \& Lieb \cite{MR2925386}]
  Let $K$ be integrable on the unit ball in $\C$.
  Then the operator $\S^{2n+1}$ with kernel $K(\zeta\cdot\bar\eta)$ is diagonal with respect to the decomposition \eqref{eq:l2-cr-spherical-harmonic-decomposition}, and on $\mcH_{j,k}$ its eigenvalue is given by
  \begin{align*}
    &C\int_{-1}^{1}(1-t)^{n-1}(1+t)^{\frac{|j-k|}{2}}P_{m}^{(n-1,|j-k|)}(t)\int_{-\pi}^{\pi}{K\left( e^{-i\vartheta}\sqrt{\frac{1+t}{2}} \right)}e^{i(j-k)\vartheta}d\vartheta dt\\
    &C=\frac{m!n!}{2\pi2^{n+\frac{|j-k|}{2}}(m+n-1)!}\\
    &m=\min\left\{ j,k \right\},
  \end{align*}
  where $P_{m}^{(\a,\b)}$ is a Jacobi polynomial.
  \label{prop:frank-lieb-funk-hecke-formula}
\end{proposition}
In particular, they showed that for $-1<\a<(n+1)/2$ the eigenvalue of the operator with kernel $|1-\zeta\cdot\bar\eta|^{-2\a}$ on $\mcH_{j,k}$ is
\begin{equation}
  E_{j,k}^{\a}=\frac{n!\Gamma(n+1-2\a)}{\Gamma^{2}(\a)}\frac{\Gamma(j+\a)}{\Gamma(j+n+1-\a)}\frac{\Gamma(k+\a)}{\Gamma(k+n+1-\a)}.
  \label{eq:funk-hecke-eigenvalue-for-specified-kernel}
\end{equation}

\subsection{Some Special Function Theory}

Lastly, we recall some properties about hypergeometric functions and two classical integral identities.
See the handbook \cite{MR0167642} for more details.
This material is relevant only to Section \ref{sec:scattering-problem-solution2}.

A hypergeometric function is given as a series of the form
\begin{equation}
  F(a,b,c;z):=\sum_{k=0}^{\oo}\frac{(a)_{k}(b)_{k}}{(c)_{k}}\frac{z^{k}}{k!}=\frac{\Gamma(c)}{\Gamma(a)\Gamma(b)}\sum_{k=0}^{\oo}\frac{\Gamma(a+k)\Gamma(b+k)}{\Gamma(c+k)}\frac{z^{k}}{k!}.
  \label{eq:hypergeoemtric-definition}
\end{equation}
Here,
\begin{align*}
  (a)_{k} &=
  \begin{cases}
    1 & k=0\\
    a(a+1)\cdots(a+k-1) & k>0
  \end{cases}\\
  &= \frac{\Gamma(a+1)}{\Gamma(a-k+1)}
\end{align*}
denotes the rising Pochhammer symbol of $a$.
A hypergeometric function $F(a,b,c;z)$ satisfies the hypergeometric differential equation
\begin{equation}
  z(1-z)y''+(c-(a+b+1)z)y'-aby=0.
  \label{eq:hypergeometric-differential-equation}
\end{equation}
Moreover, hypergeometric functions enjoy the following three properties
\begin{equation}
  \begin{aligned}
    F(a,b,c;1)&=\frac{\Gamma(c)\Gamma(c-a-b)}{\Gamma(c-a)\Gamma(c-b)}\text{ provided }\Re(c-a-b)>0\\
    F(a,b,c;z)&=(1-z)^{c-a-b}F(c-a,c-b,c;z)\\
    \frac{d^{k}}{dz^{k}}F(a,b,c;z)&=\frac{(a)_{k}(b)_{k}}{(c)_{k}}F(a+k,b+k,c+k;z)
  \end{aligned}.
  \label{eq:hypergeometric-identities-list}
\end{equation}

We will also make use of two classical integral identities.
Firstly, we recall the following cosine integral:
\begin{equation}
  \begin{aligned}
    \int_{-\pi}^{\pi}d\vartheta(1-2r\cos\vartheta+r^{2})^{-\a}e^{i(j-k)\vartheta}&=\frac{2\pi}{\Gamma^{2}(\a)}\sum_{\mu=0}^{\oo}r^{|j-k|+2\mu}\frac{\Gamma(\a+\mu)\Gamma(\a+|j-k|+\mu)}{\mu!\left( |j-k|+\mu \right)!}
  \end{aligned},
  \label{eq:cosine-integral}
\end{equation}
which holds for $0 \leq r <1$.
Lastly, we recall the following Jacobi polynomial integral
\begin{equation}
  \begin{aligned}
    \int_{-1}^{1}dt(1-t)^{n-1}(1+t)^{|j-k|+\mu}&P_{m}^{(n-1,|j-k|)}(t)\\
    =&\begin{cases}
      0&\text{if }\mu<m\\
      2^{|j-k|+n+\mu}\frac{\mu!}{m!(\mu-m)!}\frac{(|j-k|+\mu)!(m+n-1)!}{(|j-k|+m+n+\mu)!}&\text{if }\mu\geq{m}
    \end{cases}.
  \end{aligned}
  \label{eq:jacobi-integral}
\end{equation}

\section{Boundary Operators $B_{2 j}^{2\g}$ $(0\leq j\leq \lfloor\gamma/2\rfloor)$ and $B_{2 j+2[\gamma]}^{2\g}$ $(0\leq j\leq \lfloor\gamma\rfloor-\lfloor\gamma/2\rfloor-1)$}
\label{sec:boundary-operators}

In this section we define the appropriate boundary operators which arise naturally in the CR trace inequality of order $2k$, for $k \in \N_{>0}$.
We then establish some of their properties.

Recall that, for $\g \in (0,n+1)\setminus\N$, we let $\lfloor \g \rfloor$, $[\g] = \g - \lfloor \g \rfloor$ and $k = \lfloor \g \rfloor + 1$.
Recall also that $\tilde\Delta_{\B} = \Delta_{\B} + (n+1)^{2}$
and
\[
  \mathcal{C}^{2\g}(\mcU^{n+1}) = C_{\operatorname{even}}^{\oo} ( \overline{\mcU^{n+1}}) + \rho^{2 [\g]} C_{\operatorname{even}}^{\oo}( \overline{\mcU^{n+1}}),
\]
where $f \in C_{\operatorname{even}}^{\oo}(\overline{\mcU^{n+1}})$ indicates that $f$ has a Taylor expansion in terms of even powers of $\rho$.
Let $\dot H^{k,\g}(\mcU^{n+1})$ denote the closure of $C_{0}^{\oo}(\overline{\mcU^{n+1}})$ with respect to $\mathcal{E}_{2\g}(U)^{1/2} = \mathcal{Q}_{2\g}(U,U)^{1/2}$.
Let $S^{\g}(\mathbb{H}^{n})$ denote the usual Folland-Stein space as given in \cite{MR367477}.
Whenever appropriate integrability is needed, we assume functions on $\overline{\mcU^{n+1}}$ to belong to $\mcC^{2\g}(\mcU^{n+1}) \cap \dot H^{k,\g}(\mcU^{n+1})$ and boundary functions to belong to $S^{\g}(\mathbb{H}^{n})$.


We now state and prove properties of the boundary operators.
We will need the following elementary lemma.
\begin{lemma}
  Let $f \in C^{\oo}(\p \mathcal{U}^{n+1})$ and $\a \in \R$.
  Then
  \begin{align*}
    \Delta_{\B}  (\rho^{\a}f) &=\rho^{\a}\left[ \a(\a-2-2n) f + 2\rho^{2}\Delta_{b}f + \rho^{4} \p_{tt}f\right].
  \end{align*}
  In particular,  there holds
  \begin{equation}
    \tilde\Delta_{\B} (\rho^{n+1-\g}f)  = \rho^{n+1-\g}\left[ \g^{2} f  + 2 \rho^{2 } \Delta_{b} f + \rho^{4 } \p_{tt} f\right],
    \label{eq:laplace-y-2m-on-rhof}
  \end{equation}
  i.e.
  \begin{equation}
    \left(\tilde\Delta_{\B}-\gamma^2\right) (\rho^{n+1-\g}f)  = \rho^{n+3-\g}\left[  2  \Delta_{b} f + \rho^{2 } \p_{tt} f\right].
    \label{eq:laplace-y-2m-on-rhof2}
  \end{equation}

  \label{lem:laplace+c-acting-on-functions}
\end{lemma}

\begin{proof}
  Straight forward computation.
\end{proof}

By using (\ref{eq:laplace-y-2m-on-rhof2}), we have the following
corollary:
\begin{corollary}\label{co2.2}
  Let $i<j$. If $f\in \rho^{2i}\ceven$, then
  \begin{align}\label{corollary2.21}
    \prod_{\ell=i}^{j-1}\left( \tilde\Delta_{\mathbb{B}} - (\g -2\ell)^{2} \right)(\rho^{n+1-\gamma}f)\in \rho^{n+1-\gamma+2j}\ceven.
  \end{align}
  If $f\in \rho^{2i+2[\gamma]}\ceven$, then
  \begin{align}\nonumber
    &\prod_{\ell=i}^{j-1}\left( \tilde \Delta_{\mathbb{B}} - (\g + 2\ell - 2\lfloor \g \rfloor )^{2}\right)(\rho^{n+1-\gamma}f)\\
    \label{corollary2.22}
    =&\prod_{\ell=i}^{j-1}\left( \tilde \Delta_{\mathbb{B}} - (\g -2\ell - 2[ \g ] )^{2}\right)(\rho^{n+1-\gamma}f)\in \rho^{n+1-\gamma+2j+2[\gamma]}\ceven.
  \end{align}

\end{corollary}
\begin{proof}
If $f\in \rho^{2i}\ceven$, then  $f=\rho^{2i}g$ with $g\in \ceven$. By using (\ref{eq:laplace-y-2m-on-rhof2}), we have
\begin{align*}
\left( \tilde\Delta_{\mathbb{B}} - (\g -2i)^{2} \right)(\rho^{n+1-\gamma}f)=\left( \tilde\Delta_{\mathbb{B}} - (\g -2i)^{2} \right)(\rho^{n+1+2i-\gamma}g)\in
\rho^{2i+2}\ceven.
\end{align*}
Therefore,  by induction, one can easily get (\ref{corollary2.21})

Similarly, if $f\in \rho^{2i+2[\gamma]}\ceven$, then by using the identity $\g + 2\ell - 2\lfloor \g \rfloor=-\gamma+2\ell+2[\gamma]$ and (\ref{eq:laplace-y-2m-on-rhof2}), one  can easily get (\ref{corollary2.22}). These complete the proof of Corollary \ref{co2.2}.

\end{proof}

\begin{lemma}\label{lm1.2}
  For any $j \in \N_{>0}$, set
  \begin{align*}
    \tilde B_{2j}^{2\g } &= \rho^{-(n+1) + \g - 2j}  \circ \prod_{\ell=0}^{j-1}\left( \tilde\Delta_{\mathbb{B}} - (\g -2\ell)^{2} \right) \left( \tilde \Delta_{\mathbb{B}} - (\g + 2\ell - 2\lfloor \g \rfloor )^{2}\right) \circ  \rho^{n+1-\g } |_{\rho=0};\\
    \tilde B_{2j + 2[\g]}^{2\g}
    &= \rho^{-(n+1) + \g - 2j  - 2[\g]}\circ \prod_{\ell=0}^{j} \left( \tilde\Delta_{\mathbb{B}} - (\g -2\ell)^{2} \right) \prod_{\ell=0}^{j - 1} \left( \tilde \Delta_{\mathbb{B}} - (\g  + 2\ell - 2\lfloor \g \rfloor )^{2} \right) \circ \rho^{n+1-\g}  |_{\rho=0}.
  \end{align*}
  Then
  \begin{align*}
    b_{2j} := \tilde B_{2j}^{2\g } (\rho^{2j})=&4^{2j}j!\frac{\Gamma(j+1-[\gamma])}{\Gamma(1-[\gamma])}
    \cdot\frac{\Gamma(\gamma+1-j)}{\Gamma(\gamma+1-2j)}\cdot
    \frac{\Gamma(\lfloor \g \rfloor+1-j)}{\Gamma(\lfloor \g \rfloor+1-2j)};\\
    b_{2j+2[\g]} :=  \tilde B_{2j+2[\g]}^{2\g } (\rho^{2j+2[\g]})=&-4^{2j+1}j!\frac{\Gamma(j+1+[\gamma])}{\Gamma([\gamma])}
    \cdot\frac{\Gamma(\lfloor\gamma\rfloor+1-j)}{\Gamma(\lfloor\gamma\rfloor-2j)}\cdot
    \frac{\Gamma(\lfloor \g \rfloor+1-j-[\gamma])}{\Gamma(\lfloor \g \rfloor+1-2j-[\gamma])}.
  \end{align*}
\end{lemma}
\begin{proof}
  By using (\ref{eq:laplace-y-2m-on-rhof}), we have
  \begin{align*}
    \tilde\Delta_{\B} \rho^{n+1-\g}  = \g^{2} \rho^{n+1-\g}.
  \end{align*}
  Therefore,
  \begin{align*}
    b_{2j} =&\rho^{-(n+1) + \g - 2j}   \prod_{\ell=0}^{j-1}\left( \tilde\Delta_{\mathbb{B}} - (\g -2\ell)^{2} \right) \left( \tilde \Delta_{\mathbb{B}} - (\g + 2\ell - 2\lfloor \g \rfloor )^{2}\right)  \rho^{n+1-\g +2j} |_{\rho=0}\\
    =&\prod_{\ell=0}^{j-1}\left( (\gamma-2j)^{2} - (\g -2\ell)^{2} \right) \left(
    (\gamma-2j)^{2} - (\g + 2\ell - 2\lfloor \g \rfloor )^{2} \right)\\
    =&4^{2j}\prod_{\ell=0}^{j-1}(\ell-j)(\gamma-j-\ell)(\lfloor\gamma\rfloor-\ell-j)(\gamma+\ell-j-\lfloor\gamma\rfloor)\\
     =&4^{2j}\prod_{\ell=0}^{j-1}(j-l)(\gamma-j-\ell)(\lfloor\gamma\rfloor-\ell-j)(j-\ell-[\gamma])\\
    =&4^{2j}j!\frac{\Gamma(\gamma+1-j)}{\Gamma(\gamma+1-2j)}\cdot
    \frac{\Gamma(\lfloor \g \rfloor+1-j)}{\Gamma(\lfloor \g \rfloor+1-2j)}\cdot\frac{\Gamma(j+1-[\gamma])}{\Gamma(1-[\gamma])}.
  \end{align*}
  Similarly,
  \begin{align*}
    b_{2j+2[\g]} =&\prod_{\ell=0}^{j}\left( (\gamma-2j-2[\gamma])^{2} - (\g -2\ell)^{2} \right) \prod_{\ell=0}^{j-1}\left(
    (\gamma-2j-2[\gamma])^{2} - (\g + 2\ell - 2\lfloor \g \rfloor )^{2} \right)\\
    =&4^{2j+1}\prod_{\ell=0}^{j}(\ell-j-[\gamma])(\lfloor\gamma\rfloor-j-\ell)\prod_{\ell=0}^{j-1}(\ell-j)(\lfloor\gamma\rfloor-\ell-j-[\gamma])\\
    =&-4^{2j+1}j!\frac{\Gamma(j+1+[\gamma])}{\Gamma([\gamma])}
    \cdot\frac{\Gamma(\lfloor\gamma\rfloor+1-j)}{\Gamma(\lfloor\gamma\rfloor-2j)}\cdot
    \frac{\Gamma(\lfloor \g \rfloor+1-j-[\gamma])}{\Gamma(\lfloor \g \rfloor+1-2j-[\gamma])}.
  \end{align*}

\end{proof}

\begin{lemma}
  \begin{enumerate}
    \item If $f \in \rho^{2j} \ceven$, then
      \begin{equation}
	B_{2j}^{2\g}f = (\rho^{-2j}f)|_{\rho=0},\;\;\; 0\leq j\leq \lfloor\gamma/2\rfloor.
	\label{lem:boundary-operator-on-y2j-f}
      \end{equation}
    \item If $f \in \rho^{2j+2[\g]} \ceven$, then
      \begin{equation}
	\begin{aligned}
	  B_{2j + 2[\g]}^{2\g}f &= (\rho^{-2j-2[\g]}f)|_{\rho=0},\;\;\;0\leq j\leq \lfloor\gamma\rfloor-\lfloor\gamma/2\rfloor-1.
	\end{aligned}
	\label{lem:boundary-operator-on-y2j2g0-f}
      \end{equation}
  \end{enumerate}
  \label{lem:example-computations-of-B}
\end{lemma}

\begin{proof}
  If $f \in \rho^{2j} \ceven$, then there are boundary functions  $f_{\a} \in C^{\oo}(\p\mcU^{n+1})$ such that
  \[
    f = \sum_{m=0}^{\oo} \rho^{2j + 2m} f_{2j+2m}.
  \]
  Then \eqref{lem:boundary-operator-on-y2j-f} follows easily from the definition of $B_{2j}^{2\g}$ (see (\ref{1.18}))  and   a repeated use of \eqref{eq:laplace-y-2m-on-rhof}.
  The identity \eqref{lem:boundary-operator-on-y2j2g0-f} follows similarly.
\end{proof}

\begin{lemma}\label{lem:some-kernel-elements-of-B}
  Let $f \in C^{\oo}(\p \mathcal{U}^{n+1})$ and $k \in \N_{\geq0}$.
  Then
  \begin{align*}
    B_{2j+2[\g]}^{2\g}(\rho^{2k}f) = &0,\;\;\; 0\leq j\leq \lfloor\gamma/2\rfloor;\\
    B_{2j}^{2\g}(\rho^{2[\g]+2k}f) =& 0,\;\;\; 0\leq j\leq \lfloor\gamma\rfloor-\lfloor\gamma/2\rfloor-1.
  \end{align*}
\end{lemma}
\begin{proof}
  We only show the case $k=0$ and the others are similar.
  By Corollary \ref{co2.2},
  \begin{align*}
    \prod_{\ell=0}^{j}\left( \tilde\Delta_{\mathbb{B}} - (\g -2\ell)^{2} \right)(\rho^{n+1-\gamma}f)\in \rho^{n+2j+3-\gamma}C^{\oo}(\p \mathcal{U}^{n+1}).
  \end{align*}
Therefore,
  \begin{align*}
    \rho^{-(n+1) + \g - 2j  - 2[\g]}  \prod_{\ell=0}^{j - 1} \left( \tilde \Delta_{\mathbb{B}} - (\g  + 2\ell - 2\lfloor \g \rfloor )^{2} \right)
     \prod_{\ell=0}^{j} \left( \tilde\Delta_{\mathbb{B}} - (\g -2\ell)^{2} \right)( \rho^{n+1-\g}f)\in \rho^{2-2[\gamma]}C^{\oo}(\p \mathcal{U}^{n+1}).
  \end{align*}
By the definition of $B_{2j+2[\g]}^{2\g}$ (see (\ref{1.19})), we have $B_{2j+2[\g]}^{2\g}f = 0$.

  The proof of $B_{2j}^{2\g}(\rho^{2[\g]}f) = 0$ is similar and we omit it.



\end{proof}

Lemmas \ref{lem:example-computations-of-B} and \ref{lem:some-kernel-elements-of-B} allow us to write $f \in \mathcal{C}^{2\g}$ as a series in terms of the boundary operators.
This is detailed in the following corollary.

\begin{corollary}
  If $f \in \mathcal{C}^{2\g}$, then
  \[
    f = \sum_{j=0}^{\oo} \rho^{ 2j} f _{ 2j} + \sum_{j=0}^{\oo} \rho^{2j + 2[\g]} f_{ 2j + 2[\g] },
  \]
  for some boundary functions $f_{\a} \in C^{\oo}(\p\mcU^{n+1})$ given by
  \begin{align*}
    f_{ 2j } &= B_{ 2j }^{2\g}\left( f - \sum_{m=0}^{j-1} \rho^{ 2m } f_{ 2m } \right),\;\;\;\;\;\; \;\;\;\;\;\;\;\;\;\;\;\;\;\;0\leq j\leq \lfloor\gamma/2\rfloor;\\
    f_{ 2j + 2[\g]} &= B_{ 2j + 2[\g]}^{2\g}\left( f - \sum_{m=0}^{j-1} \rho^{ 2m + 2[\g]} f_{ 2m + 2[\g]} \right),\;\;\; 0\leq j\leq \lfloor\gamma\rfloor-\lfloor\gamma/2\rfloor-1.
  \end{align*}
  \label{cor:general-series-expansion-in-terms-of-B}
\end{corollary}

\begin{proof}
  Since $f \in \mathcal{C}^{2\g}$, there are functions $\left\{ f_{2j},f_{2j+2[\gamma]} \right\} \subset C^{\oo}(\p\mcU^{n+1})$ such that
  \[
    f = \sum_{j=0}^{\oo} \rho^{2j}f_{2j} + \sum_{j=0}^{\oo} \rho^{2j+2[\g]}f_{2j + 2[\g]}.
  \]
  Evidently, Lemmas \ref{lem:example-computations-of-B} and \ref{lem:some-kernel-elements-of-B} give
  \[
    B_{0}^{2\g} f = f_{0} \text{ and } B_{2[\g]}^{2\g} f = f_{2[\g]}.
  \]
  Similarly,
  \[
    B_{2}^{2\g}(f - B_{0}^{2\g} f)  = f_{2} \text{ and } B_{2[\g] + 2}^{2\g} (f - \rho^{2[\g]} B_{2[\g]}^{2\g}f) = f_{2 + 2[\g]}.
  \]
  Proceeding inductively gives the desired series expansion.

\end{proof}

\section{Dirichlet Problem and Scattering in the CR setting}
\label{sec:dirichlet-problem-and-scattering}

Scattering theory was initially started by Mazzeo and Melrose \cite{MM} and was further developed by Graham and Zworski \cite{MR1965361}
to derive the fractional GJMS operators based on the construction of the ambient space of C. Fefferman and Graham \cite{FeffermanGr2, FeffermanGr}.
Applying to the CR setting, we refer the reader to \cite{Frank2} for more detailed expositions.

Fix $\g \in (0,n+1)\setminus\N$ and set $k = \lfloor \g \rfloor + 1$.
Let $S^{\g,2}(\p\mathcal{U}^{n+1})$ be the usual Folland-Stein space.
We recall some scattering theory for complex manifolds with CR boundary; see \cite{MR2472889,epstein} for more details.
Given $f \in C^{\oo}(\p \mathcal{U}^{n+1}) \cap S^{\g,2}(\p\mathcal{U}^{n+1})$ and $s \in \C$ with $\Re s > (n+1)/2$, there is a unique solution $\mathcal{P}(s)f$ of the Poisson equation
\begin{equation}
  \Delta_{\B} u + 4s(n+1-s)u = 0
  \label{eq:scattering-equation}
\end{equation}
such that
\begin{equation}
  \begin{cases}
    \mathcal{P}(s)f = q^{n+1-s}F + q^{s}G & \text{ for some } F,G \in C_{\operatorname{even}}^{\oo} ( \overline{\mcU^{n+1}})\\
    F|_{\p\mathcal{U}^{n+1}} = f\\
  \end{cases}.
  \label{eq:scattering-equation-expansion}
\end{equation}
Unless otherwise specified, take $s = \frac{n+1+\g}{2}$.
Defining the scattering operator $S(s):f  \mapsto  G|_{\p\mathcal{U}^{n+1}}$, we recall that the operator $S$ defines the CR fractional sub-Laplacian $P_{\g}$ on $\p\mcU^{n+1}$ by
\begin{equation}
  P_{\g} f = c_{\g}S(s)f,\;\;c_{\g} = 2^{\g} \frac{\Gamma(\g)}{\Gamma(-\g)}.
  \label{eq:fractional-boundary-operator}
\end{equation}
It had been shown in \cite{Frank2} that for $0<\gamma<1$, it holds
\[
  P_{\g} = c_{\g} S\left(\frac{n+1+\g}{2}\right)f=(2|T|)^{\gamma}\frac{\Gamma(\frac{1+\gamma}{2}+\frac{-\Delta_{b}}{2|T|})}
{\Gamma(\frac{1-\gamma}{2}+\frac{-\Delta_{b}}{2|T|})}f.
\]
We remark that it is also valid for
$\g \in (0,n+1)\setminus\N$ and the proof is  similar.
Therefore, by (\ref{eq:scattering-equation-expansion}), we obtain (here we set  $q=\frac{1}{2}\rho^{2}$)
\begin{align}\label{b3.4}
  \mathcal{P}(\frac{n+1+\gamma}{2})f=&\big(\frac{1}{2}\rho^{2}\big)^{\frac{n+1-\gamma}{2}}F_j+\big(\frac{1}{2}\rho^{2}\big)^{\frac{n+1+\gamma}{2}}G_j\\
  =&\big(\frac{1}{2}\rho^{2}\big)^{\frac{n+1-\gamma}{2}}(f+\rho^{2} f_{2}+\cdots)+\nonumber \\
  \label{b3.5}
  &\frac{1}{c_{\gamma}}\big(\frac{1}{2}\rho^{2}\big)^{\frac{n+1+\gamma}{2}}\left[ P_{\g} f +\rho^{2}f_{2+2[\gamma]}+\cdots\right]
\end{align}
for some boundary functions $f_{\a} \in C^{\oo}(\p\mcU^{n+1})$.

We are interested in solving the following Dirichlet boundary value problem for the $L_{2k}$:
\begin{equation}
  \begin{cases}
    L_{2k} V = 0 & \text{in }\mcU^{n+1}\\
    B_{2j}^{2\g}(V) = f^{(2j)},& 0 \leq j \leq [\g/2]\\
    B_{2j+2[\g]}^{2\g}(V)  = \phi^{(2j)}, & 0 \leq j \leq [\g] - [\g/2] -1
  \end{cases}.
  \label{eq:boundary-value-dirichlet-problem-l2k}
\end{equation}
Here, $f^{(2j)} \in C^{\oo}(\p\mathcal{U}^{n+1}) \cap S^{\g-2j,2}(\p\mathcal{U}^{n+1})$ and $\phi^{(2j)} \in C^{\oo}(\p\mathcal{U}^{n+1}) \cap S^{\lfloor \g \rfloor - [\g] - 2j}(\p\mathcal{U}^{n+1})$ are the boundary data.

\begin{theorem}
  Let $\g \in (0,\oo) \setminus \N$ and fix boundary data
  \begin{align*}
    f^{(2j)} \in C^{\oo}(\p\mathcal{U}^{n+1}) \cap S^{\g-2j,2}(\p\mathcal{U}^{n+1}) \qquad   &\text{ for } \qquad 0 \leq j \leq \lfloor \g/2 \rfloor\\
    \phi^{(2j)}  \in C^{\oo}(\p\mathcal{U}^{n+1}) \cap S^{\lfloor \g \rfloor - [\g] -2j,2}(\p\mathcal{U}^{n+1}) \qquad & \text{ for } \qquad 0 \leq j \leq \lfloor \g \rfloor - \lfloor \g/2 \rfloor - 1.
  \end{align*}
  Then the Dirichlet problem
  \begin{equation}
    \begin{cases}
      L_{2k} V = 0, & \text{in }\mcU^{n+1}\\
      B_{2j}^{2\g}(V) = f^{(2j)},& 0 \leq j \leq \lfloor \g/2 \rfloor\\
      B_{2j+2[\g]}^{2\g}(V)  = \phi^{(2j)}, & 0 \leq j \leq \lfloor \g \rfloor - \lfloor \g/2 \rfloor -1
    \end{cases}.
    \label{eq:boundary-value-dirichlet-problem-l2k-copied}
  \end{equation}
  has a unique solution given by
  \begin{align}\nonumber
    V =& \sum_{j=0}^{\lfloor \g/2 \rfloor} 2^{\frac{n+1-\g+2j}{2}}\rho^{-(n+1)+\g} \mathcal{P}\left( \frac{n+1+\g-2j}{2} \right)f^{(2j)}  + \\
    &\sum_{j=0}^{\lfloor \g \rfloor - \lfloor \g/2 \rfloor - 1}2^{\frac{n+1 - \lfloor \g \rfloor + [\g] + 2j}{2} } \rho^{-(n+1) + \g} \mathcal{P}\left( \frac{n+1 + \lfloor \g \rfloor - [\g] - 2j}{2} \right) \phi^{(2j)}.
    \label{eq:V-expansion-in-terms-of-f-and-phi}
  \end{align}
  \label{thm:solution-to-l2k-dirichlet-problem}
\end{theorem}

\begin{proof}
 By (\ref{1.13b}) and the definition of $\mathcal{P}$, we have  $L_{2k} V = 0$.

  Next let
  \begin{align*}
    V =&  \sum_{j=0}^{\lfloor \g/2 \rfloor} 2^{\frac{n+1-\g+2j}{2}}\rho^{-(n+1)+\g} \mathcal{P}\left( \frac{n+1+\g-2j}{2} \right)f^{(2j)}  + \\
    &\sum_{j=0}^{\lfloor \g \rfloor - \lfloor \g/2 \rfloor - 1}2^{\frac{n+1 -\lfloor \g \rfloor + [\g] + 2j}{2} } \rho^{-(n+1) + \g} \mathcal{P}\left( \frac{n+1 + \lfloor \g \rfloor - [\g] - 2j}{2} \right) \phi^{(2j)}\\
    =&:V_{1} + V_{2}.
  \end{align*}
  To see that $V$ satisfies the boundary conditions, first observe
  \begin{align*}
  2^{\frac{n+1-\gamma+2j}{2}}\rho^{-(n+1) + \g} \mathcal{P}\left( \frac{n+1+\g-2j}{2} \right)f^{(2j)} = \rho^{2j}F_{j} + 2^{-\gamma+2j}\rho^{2\g-2j}G_{j}
  \end{align*}
    with $F_{j}|_{\rho=0} = f^{(2j)}$.
  As is standard, we may write $F_{j}$ as a $\ceven$ function$\mod O(\rho^{\oo})$ and whence
  \begin{equation}
    2^{\frac{n+1-\gamma+2j}{2}}\rho^{-(n+1) + \g} \mathcal{P}\left( \frac{n+1+\g-2j}{2} \right)f^{(2j)} = \rho^{2j}f^{(2j)} + \rho^{2j+2}f_{2} + \cdots \mod O(\rho^{\oo})
    \label{eq:expansion-mod-rho-infinity-for-f2j}
  \end{equation}
  for some $f_{2} \in C^{\oo}(\p\mcU^{n+1})$.
  We may similarly expand
  \[
    2^{\frac{n+1 - \lfloor \g \rfloor + [\g] + 2j}{2} } \rho^{-(n+1) + \g} \mathcal{P}\left( \frac{n+1+\lfloor \g \rfloor - [\g] - 2j}{2} \right) \phi^{(2j)}.
  \]
  Now, observe that $B_{2j}^{2\g}$ annihilates all terms in the expansion \eqref{eq:V-expansion-in-terms-of-f-and-phi} except for
  \[
    2^{\frac{n+1-\gamma+2j}{2}}\rho^{-(n+1)+\g}\mathcal{P}\left( \frac{n+1+\g-2j}{2} \right)f^{(2j)}.
  \]
  This is straightforward to see: either apply Lemmas \ref{lem:example-computations-of-B} and \ref{lem:some-kernel-elements-of-B} or use that $B_{2j}^{2\g}$ contains as a factor those operators of the form $\tilde\Delta_{\B} - \a^{2}$ which annihilate
  \[
    \mathcal{P}\left( \frac{n+1+\g-2m}{2} \right)f^{(2m)} \quad \text{ or } \quad \mathcal{P}\left( \frac{n+1+\lfloor \g \rfloor - [\g] - 2m'}{2} \right)\phi^{(2j)}
  \]
  for $m \neq j$ and $m' = 0 ,\ldots, \lfloor \g \rfloor - \lfloor \g/2\rfloor - 1$.
  Therefore, using \eqref{eq:expansion-mod-rho-infinity-for-f2j} and Lemma \ref{lem:example-computations-of-B}, we conclude
  \[
    B_{2j}^{2\g} (V) = f^{(2j)}
  \]
  We similarly conclude
  \[
    B_{2j + 2[\g]}^{2\g}(V) = \phi^{(2j)}.
  \]

  We now show that the solution is unique.
  Let $V$ be as in the theorem conclusion.
  First note that, if $U \in C^{2\g}$ solves
  \[
    \left( \tilde\Delta_{\B} - \g^{2} \right)(\rho^{n+1-\g}U)=0
  \]
  with $U|_{\rho=0} = 0$, then $U = 0$ by uniqueness of the scattering problem.
  Now consider the case $\lfloor \g \rfloor = 2P \in 2\N$.
  Then, by indexing the expression for $L_{2k}^{+}$, we see \eqref{eq:boundary-value-dirichlet-problem-l2k-copied} is equivalent to the problem
  \[
    \begin{cases}
      \left( \tilde\Delta_{\B} - (\g - 2P)^{2} \right) \prod\limits_{\ell=0}^{P-1}\left( \tilde\Delta_{\B}  - (\g - 2\ell)^{2} \right)\left( \tilde\Delta_{\B} - \left( \g +2 \ell - 2\lfloor \g \rfloor \right)^{2} \right)(\rho^{n+1-\g}V) = 0\\
      B_{2j}^{2\g}(V) = f^{(2j)} ,\qquad 0 \leq j \leq \lfloor \g/2 \rfloor \\
      B_{2j+2[\g]}^{2\g} (V) = \phi^{(2j)} ,\qquad 0 \leq j \leq \lfloor \g \rfloor - \lfloor \g /2 \rfloor -1
    \end{cases}.
  \]
  Let $V'$ be another solution and note $B_{2P}^{2\g}(V) =  B_{2P}^{2\g} (V')$.
  Setting
  \[
    U = \rho^{-(n+1) + \g - 2P} \prod_{\ell=0}^{P-1}\left( \tilde\Delta_{\B}  - (\g - 2\ell)^{2} \right)\left( \tilde\Delta_{\B} - \left( \g +2 \ell - 2\lfloor \g \rfloor \right)^{2} \right)(\rho^{n+1-\g}(V-V')),
  \]
  we have
  \[
    \left( \tilde\Delta_{\B} - (\g-2P)^{2} \right)(\rho^{n+1-\g+2P}U) = 0
  \]
  and
  \[
    U|_{\rho=0} = B_{2P}^{2\g}(V-V') = 0.
  \]
  By the observations above, we conclude $U=0$ and so both $V-V'$ solves the problem
  \[
    \begin{cases}
      \prod\limits_{\ell=0}^{P-1}\left( \tilde\Delta_{\B}  - (\g - 2\ell)^{2} \right)\left( \tilde\Delta_{\B} - \left( \g +2 \ell - 2\lfloor \g \rfloor \right)^{2} \right)(\rho^{n+1-\g}(V-V')) = 0\\
      B_{2j}^{2\g}(V-V') = 0 \\
      B_{2j+2[\g]}^{-1} B_{2j+2[\g]}^{2\g} (V-V') = 0
    \end{cases}.
  \]
  Proceeding inductively gives $V=V'$ and hence uniqueness.
  The case $\lfloor \g \rfloor \in 2\N + 1$ is handled similarly.
\end{proof}

\begin{remark}
  \label{remark:dirichlet-problem-with-boundary-data-given-by-U}
  Given $V \in \mathcal{C}^{2\g}$, let $\tilde{V}$ denote the solution to
  \begin{equation}
    \begin{cases}
      L_{2k} \tilde{V} = 0 & \text{ in } \mcU^{n+1}\\
      B_{2j}^{2\g}(\tilde V) = B_{2j}^{2\g} (V) & 0 \leq j \leq \lfloor \g/2 \rfloor\\
      B_{2j + 2[\g]}^{2\g} (\tilde V) = B_{2j+2[\g]}^{2\g} (V) & 0 \leq j \leq \lfloor \g \rfloor - \lfloor \g/2 \rfloor - 1
    \end{cases}.
    \label{eq:dirichlet-problem-with-boundary-data-given-by-U}
  \end{equation}
  Then $V - \tilde V$ satisfies
  \begin{align*}
    B_{2j}^{2\g} (V - \tilde V) &= 0, \quad 0 \leq j \leq \lfloor \g/2 \rfloor\\
    B_{2j+2[\g]}^{2\g} (V- \tilde V) & = 0 , \quad 0 \leq j \leq \lfloor \g \rfloor - \lfloor \g/2 \rfloor - 1
  \end{align*}
  and whence
  \begin{align*}
    V - \tilde V \in \rho^{2\lfloor \g/2 \rfloor +2} \ceven + \rho^{2(\g - \lfloor \g/2 \rfloor)  } \ceven.
  \end{align*}
\end{remark}

\section{Proof of Theorem \ref{thm:boundary-to-fractional-operator} and  \ref{thm:boundary-operator-conformal-covariance}  }\label{Section4}

We first define boundary operators for the rest of $j$.
For notational convenience and consistency, we introduce the following notation:
\begin{align}\nonumber
  \Pi^{\g} =& (-1)^{\lfloor \g \rfloor +1}\prod_{i=0}^{\lfloor \g \rfloor} \left( \tilde \Delta_{\B} - (\g -2i)^{2} \right)\\
  =&
  (-1)^{\lfloor \g \rfloor +1}\prod_{l=0}^{\lfloor \g \rfloor} \left( \tilde \Delta_{\B} - (\g -2l-2[\gamma])^{2} \right)\;(\textrm{substituting}\;\; l=\lfloor \g \rfloor-i)\\ \label{4.1}
  \Pi_{j}^{\g} =& (-1)^{\lfloor \g \rfloor} \prod_{i=0}^{j-1} \left( \tilde \Delta_{\B} - (\g -2i)^{2} \right) \prod_{i=j+1}^{\lfloor \g \rfloor}\left( \tilde\Delta_{\B} - (\g -2i)^{2} \right)\\ \label{4.2}
  =& (-1)^{\lfloor \g \rfloor} \prod_{i=\lfloor \g \rfloor-j+1}^{\lfloor \g \rfloor} \left( \tilde \Delta_{\B} - (\g -2i-2[\g])^{2} \right) \prod_{i=0}^{\lfloor \g \rfloor-j-1}\left( \tilde\Delta_{\B} - (\g -2i-2[\gamma])^{2} \right).
\end{align}
We will also sometimes use notation such as
\begin{align*}
  \Pi_{j}^{\g} &= \frac{\Pi^{\g}}{-(\tilde \Delta_{\B} - (\g-2j)^{2})}
\end{align*}
to indicate the differential operator obtained by removing the factor $-(\tilde\Delta_{\B} -(\g-2j)^{2})$ from the expression used to define $\Pi^{\g}$.

For $V\in \mathcal{C}^{2\gamma}$, let $\tilde{V}$ be the solution of (\ref{eq:dirichlet-problem-with-boundary-data-given-by-U}). We define boundary operators for the rest of $j$ as follows:
\begin{itemize}
  \item $\lfloor\gamma/2\rfloor+1\leq j\leq\lfloor\gamma\rfloor$,
    \begin{align}
      B_{ 2j}^{2\g }(V)=&\frac{1}{b_{2j}}\rho^{-2j-n-1+\gamma}
  \Pi^{\g}_{j}(\rho^{n+1-\g}(V-\tilde V))\big|_{\rho=0}  \label{3.11}
      +\frac{2^{\gamma-2j}}{c_{2j-\gamma}}P_{2j-\gamma}B^{2\gamma}_{2\gamma-2j}(V);
    \end{align}
  \item $\lfloor\gamma\rfloor-\lfloor\gamma/2\rfloor\leq j\leq\lfloor\gamma\rfloor$,
    \begin{align}\nonumber
      B_{ 2j+2[\gamma]}^{2\g }(V)=&\frac{1}{b_{2j+2[\gamma]}}\rho^{-2j-2[\g]-n-1+\gamma}
  \Pi^{\g}_{\lfloor\gamma\rfloor-j}(\rho^{n+1-\g}(V-\tilde V))\big|_{\rho=0} \\ \label{4.4}
    &+\frac{2^{\lfloor\gamma\rfloor-2j-[\gamma]}}{c_{2j+[\gamma]-\lfloor\gamma\rfloor}}
    P_{2j+[\gamma]-\lfloor\gamma\rfloor}B^{2\gamma}_{2\lfloor\gamma\rfloor-2j}(V),
  \end{align}
\end{itemize}
where
\begin{align}\nonumber
  b_{2j} = & (-1)^{\lfloor \g \rfloor} \prod_{i\in\{0,1,\cdots,\lfloor \g \rfloor\}\setminus\{j\}}( (\gamma-2j)^{2}- (\g - 2i)^{2})\\
  \label{5.6}
  =& 4^{\lfloor\g\rfloor}j!(\lfloor\g\rfloor-j)!\frac{\Gamma(\gamma+1-j)\Gamma(j+1-[\g])}{\Gamma(\gamma+1-2j)\Gamma(2j+1-\g)},\;\; \lfloor\gamma/2\rfloor+1\leq j\leq\lfloor\gamma\rfloor; \\
  \label{5.7}
  b_{2j+2[\g]} =  & (-1)^{\lfloor \g \rfloor} \prod_{i\in\{0,1,\cdots,\lfloor \g \rfloor\}\setminus\{\lfloor\gamma\rfloor-j\}}( (\gamma-2j-2[\gamma])^{2}- (\g - 2i)^{2})
\end{align}
are also chosen   such that
\begin{align*}
  B_{2 j}^{2\g}(\rho^{2j})= B_{ 2j + 2[\g]}^{2\g }(\rho^{2j + 2[\g]})=1.
\end{align*}
We note that if we substitute  $j\rightarrow \lfloor\g\rfloor-j$ in (\ref{5.7}), then we have
\begin{align}\nonumber
 b_{2\gamma-2j}=&(-1)^{\lfloor \g \rfloor} \prod_{i\in\{0,1,\cdots,\lfloor \g \rfloor\}\setminus\{j\}}( (\gamma-2j)^{2}- (\g - 2i)^{2})\\
 \label{5.8}
 =&4^{\lfloor\g\rfloor}j!(\lfloor\g\rfloor-j)!\frac{\Gamma(\gamma+1-j)\Gamma(j+1-[\g])}{\Gamma(\gamma+1-2j)\Gamma(2j+1-\g)},\;\; 0\leq j\leq\lfloor\gamma/2\rfloor.
\end{align}

\textbf{Proof of Theorem \ref{thm:boundary-to-fractional-operator}}.   Straight forward computation.
\vspace{0.3cm}

By using (\ref{4.1})-(\ref{4.2}) and  Corollary \ref{co2.2}, we have the following corollary:
\begin{corollary}\label{lm4.1} Let $0\leq j\leq \lfloor\gamma\rfloor$.
  If $f\in \rho^{2i}\ceven$, then
  \begin{align*}
    \Pi_{j}^{\g}(\rho^{n+1-\gamma}f)&\in \rho^{n+1-\gamma+2j}\ceven,\;\; i\leq j;\\
    \Pi_{j}^{\g}(\rho^{n+1-\gamma}f)&\in \rho^{n+3-\gamma+2\lfloor\gamma\rfloor}\ceven,\;\; i>j.
  \end{align*}
  If $f\in \rho^{2i+2[\gamma]}\ceven$, then
  \begin{align*}
    \Pi_{j}^{\g}(\rho^{n+1-\gamma}f)&\in \rho^{n+1+\gamma-2j}\ceven,\;\; i\leq \lfloor\g \rfloor- j;\\
    \Pi_{j}^{\g}(\rho^{n+1-\gamma}f)&\in \rho^{n+3+\gamma}\ceven,\;\; i>\lfloor\g \rfloor- j.
  \end{align*}
\end{corollary}

By Remark \ref{remark:dirichlet-problem-with-boundary-data-given-by-U},
\begin{align*}
  V - \tilde V \in \rho^{2\lfloor \g/2 \rfloor +2} \ceven + \rho^{2(\g - \lfloor \g/2 \rfloor)  } \ceven.
\end{align*}
Therefore, using Corollary  \ref{lm4.1}, we get
\begin{itemize}
  \item $\lfloor\gamma/2\rfloor+1\leq j\leq\lfloor\gamma\rfloor$,
    \begin{align}\label{4.5}
      \Pi_{j}^{\g}(\rho^{n+1-\gamma}(  V - \tilde V) )\in & \rho^{n+1-\g+2j} \ceven + \rho^{n+3+\g  } \ceven  ;
    \end{align}
  \item $\lfloor\gamma\rfloor-\lfloor\gamma/2\rfloor\leq j\leq\lfloor\gamma\rfloor$,
    \begin{align}\label{4.6}
      \Pi_{\lfloor\g\rfloor-j}^{\g}(\rho^{n+1-\gamma}(  V - \tilde V) )\in & \rho^{n+3-\g+2\lfloor\g\rfloor} \ceven + \rho^{n+1-\g+2j+2[\g]  } \ceven,
    \end{align}
\end{itemize}
i.e.
\begin{itemize}
  \item $\lfloor\gamma/2\rfloor+1\leq j\leq\lfloor\gamma\rfloor$,
    \begin{align}\label{4.7}
      \Pi_{j}^{\g}(\rho^{n+1-\gamma}(  V - \tilde V) )\in & \rho^{n+1-\g+2j} \ceven + \rho^{n+3+\g  } \ceven  ;
    \end{align}
  \item $0\leq j\leq \lfloor\gamma/2\rfloor$,
    \begin{align}\label{4.8}
      \Pi_{j}^{\g}(\rho^{n+1-\gamma}(  V - \tilde V) )\in & \rho^{n+3+\g-2[\g]} \ceven + \rho^{n+1+\g-2j  } \ceven,
    \end{align}
\end{itemize}
Moreover, combing (\ref{4.4}) and (\ref{4.8}) yields
  \begin{align}\nonumber
      B_{ 2\gamma-2j}^{2\g }(V)=&\frac{1}{b_{2\gamma-2j}}\rho^{2j-n-1-\gamma}
  \Pi^{\g}_{j}(\rho^{n+1-\g}(V-\tilde V))\big|_{\rho=0} \\ \label{4.9}
    &+\frac{2^{2j-\gamma}}{c_{\gamma-2j}}
    P_{\gamma-2j}B^{2\gamma}_{2j}(V),\;\;\;\; 0\leq j\leq \lfloor\gamma/2\rfloor.
  \end{align}

Given another defining function $\widehat{\rho}=e^{\tau}\rho$ and if we let
$\widehat{B}_{2j}^{2\gamma}$ and $\widehat{B}_{2j+2[\gamma]}$ be the boundary operators associated with $(\mathcal{U}^{n+1}, \widehat{\rho}^{2}g_{+})$. That is,
for $V\in \mathcal{C}^{2\gamma}(\mathcal{U}^{n+1})$,
\begin{itemize}
  \item $ \widehat{B}_{ 0}^{2\g}(V)=V|_{\rho=0}$;
  \item  $1\leq j\leq \lfloor\gamma/2\rfloor$,
    \begin{align*}
      \widehat{ B}_{2 j}^{2\g}(V)=&\frac{1}{ b_{ 2j }}\widehat{\rho}^{\;-(n+1) + \g - 2j}   \prod_{\ell=0}^{j-1}D_{s-l} \prod_{\ell=\lfloor\gamma\rfloor-j+1}^{\lfloor\gamma\rfloor}D_{s-l} (\widehat{\rho}^{n+1-\g } V)|_{\widehat{\rho}=0};
    \end{align*}
  \item $0\leq j\leq \lfloor\gamma\rfloor-\lfloor\gamma/2\rfloor-1$,
    \begin{align*}
      \widehat{B}_{ 2j + 2[\g]}^{2\g }(V)=&\frac{1}{ b_{ 2j + 2[\g]}}\widehat{\rho}^{\;-(n+1) + \g - 2j  - 2[\g]} \prod_{\ell=0}^{j}D_{s-l} \prod_{\ell=\lfloor\gamma\rfloor-j+1}^{\lfloor\gamma\rfloor}D_{s-l} ( \widehat{\rho}^{n+1-\g}V)  |_{\widehat{\rho}=0};
    \end{align*}
  \item   $\lfloor\gamma/2\rfloor+1\leq j\leq\lfloor\gamma\rfloor$,
    \begin{align*}\nonumber 
      \widehat{B}_{ 2j}^{2\g }(V)
      =&\frac{1}{b_{2j}}\widehat{\rho}^{\;-2j-n-1+\gamma}
      \Pi^{\g}_{j}(\widehat{\rho}^{n+1-\g}(V-\widehat{\tilde V}))\big|_{\widehat{\rho}=0}
      +\frac{2^{\gamma-2j}}{c_{2j-\gamma}}\widehat{P}_{2j-\gamma}\widehat{B}^{2\gamma}_{2\gamma-2j}(V);
    \end{align*}
  \item   $\lfloor\gamma\rfloor-\lfloor\gamma/2\rfloor\leq j\leq\lfloor\gamma\rfloor$,
    \begin{align*}
      \widehat{B}_{ 2j+2[\gamma]}^{2\g }(V)=&\frac{1}{b_{2j+2[\gamma]}}\widehat{\rho}^{\;-2j-2[\g]-n-1+\gamma}
      \Pi^{\g}_{\lfloor\gamma\rfloor-j}(\widehat{\rho}^{n+1-\g}(V-\widehat{\tilde V}))\big|_{\widehat{\rho}=0} \\
      &+\frac{2^{\lfloor\gamma\rfloor-2j-[\gamma]}}{c_{2j+[\gamma]-\lfloor\gamma\rfloor}}
      \widehat{P}_{2j+[\gamma]-\lfloor\gamma\rfloor}\widehat{B}^{2\gamma}_{2\lfloor\gamma\rfloor-2j}(V),
    \end{align*}
\end{itemize}
where $\widehat{\tilde V}$ is the solution of
  \begin{equation}
    \begin{cases}
      \widehat{L}_{2k} \widehat{\tilde V} =\widehat{\rho}^{-(n+1) + \g - 2k}  L_{2k}^{+} (\widehat{\rho}^{n+1-\g}\widehat{\tilde V})= 0 & \text{ in } \mcU^{n+1}\\
      \widehat{B}_{2j}^{2\g}(\widehat{\tilde V}) = \widehat{B}_{2j}^{2\g} (V) & 0 \leq j \leq \lfloor \g/2 \rfloor\\
      \widehat{B}_{2j + 2[\g]}^{2\g} (\widehat{\tilde V}) = \widehat{B}_{2j+2[\g]}^{2\g} (V) & 0 \leq j \leq \lfloor \g \rfloor - \lfloor \g/2 \rfloor - 1
    \end{cases}.
    \label{5.10}
  \end{equation}

  \begin{lemma}\label{lm5.2}
Let  $V\in \mathcal{C}^{2\gamma}(\mathcal{U}^{n+1})$.
  It holds that
  \begin{align}\label{5.15}
    \widehat{B}_{2j}^{2\gamma}(V)=&e^{(-(n+1) + \g  - 2j)\tau|_{\mathbb{H}^{n}}}B_{2j}^{2\gamma}(e^{(n+1-\gamma)\tau}V),\;\;0\leq j\leq \lfloor\gamma/2\rfloor;\\
    \label{5.16}
    \widehat{B}_{2j+2[\gamma]}^{2\gamma}(V)=&e^{(-(n+1) + \g  - 2j  - 2[\g])\tau|_{\mathbb{H}^{n}}}B_{2j+2[\gamma]}^{2\gamma}(e^{(n+1-\gamma)\tau}V),\;\;0\leq j\leq \lfloor\gamma\rfloor-\lfloor\gamma/2\rfloor-1.
  \end{align}
  where $\tau|_{\mathbb{H}^{n}}$ is the restriction of $\tau$ on $\mathbb{H}^{n}$.
\end{lemma}
  \begin{proof}
  By the definition of $B_{2 j}$ and $\widehat{B}_{2 j}$, we have
    \begin{align*}
    &e^{(-(n+1) + \g  - 2j)\tau|_{\mathbb{H}^{n}}}  B_{2j}^{2\gamma}(e^{(n+1-\gamma)\tau}V)\\
    =&\frac{1}{ b_{ 2j }}e^{(-(n+1) + \g  - 2j)\tau|_{\mathbb{H}^{n}}} \rho^{-(n+1) + \g  - 2j}\prod_{\ell=0}^{j-1}D_{s-l} \prod_{\ell=\lfloor\gamma\rfloor-j+1}^{\lfloor\gamma\rfloor}D_{s-l} \left(  \rho^{n+1-\g }e^{(n+1-\gamma)\tau}V\right) |_{\rho=0}\\
    =&\frac{1}{ b_{ 2j }}\widehat{\rho}^{\;-(n+1) + \g - 2j}   \prod_{\ell=0}^{j-1}D_{s-l} \prod_{\ell=\lfloor\gamma\rfloor-j+1}^{\lfloor\gamma\rfloor}D_{s-l} (\widehat{\rho}^{n+1-\g } V)|_{\widehat{\rho}=0}\\
       =&\widehat{ B}_{2 j}^{2\g}(V).
  \end{align*}
  This proves (\ref{5.15}).
  The proof of (\ref{5.16}) is similar and we omit it.
  \end{proof}

  We claim that
  \begin{align}\label{5.11}
  \widehat{\tilde V}= e^{-(n+1-\g)\tau}  (e^{(n+1-\g)\tau}V)^{\widetilde{}},
  \end{align}
  where $(e^{(n+1-\g)\tau}V)^{\widetilde{}}$ is the solution of
  \begin{equation*}
    \begin{cases}
      L_{2k} (e^{(n+1-\g)\tau}V)^{\widetilde{}} = 0 & \text{ in } \mcU^{n+1}\\
B_{2j}^{2\g}((e^{(n+1-\g)\tau}V)^{\widetilde{}}\;) = B_{2j}^{2\g} (e^{(n+1-\g)\tau}V) & 0 \leq j \leq \lfloor \g/2 \rfloor\\
  B_{2j + 2[\g]}^{2\g} ((e^{(n+1-\g)\tau}V)^{\widetilde{}}\;) = B_{2j+2[\g]}^{2\g} (e^{(n+1-\g)\tau}V) & 0 \leq j \leq \lfloor \g \rfloor - \lfloor \g/2 \rfloor - 1
    \end{cases}.
  \end{equation*}
  In fact, by Lemma \ref{lm5.2},
  \begin{align*}
 \widehat{ B}_{2 j}^{2\g}( e^{-(n+1-\g)\tau} (e^{(n+1-\g)\tau}V)^{\widetilde{}}\;)=&e^{(-(n+1) + \g  - 2j)\tau|_{\mathbb{H}^{n}}}  B_{2j}^{2\gamma}((e^{(n+1-\g)\tau}V)^{\widetilde{}}\;)\\
 =&e^{(-(n+1) + \g  - 2j)\tau|_{\mathbb{H}^{n}}}  B_{2j}^{2\gamma}(e^{(n+1-\g)\tau}V)\\
 =&\widehat{ B}_{2 j}^{2\g}(V),\;\;
 0 \leq j \leq \lfloor \g/2 \rfloor,\\
 \widehat{B}_{2j + 2[\g]}^{2\g}(e^{-(n+1-\g)\tau} (e^{(n+1-\g)\tau}V)^{\widetilde{}}\;)=&
 e^{(-(n+1) + \g  - 2j  - 2[\g])\tau|_{\mathbb{H}^{n}}}  B_{2j+2[\g]}^{2\gamma}((e^{(n+1-\g)\tau}V)^{\widetilde{}}\;)\\
 =&e^{(-(n+1) + \g  - 2j  - 2[\g])\tau|_{\mathbb{H}^{n}}}  B_{2j+2[\g]}^{2\gamma}(e^{(n+1-\g)\tau}V)\\
 =&\widehat{B}_{2j + 2[\g]}^{2\g}(V),\;\;0 \leq j \leq \lfloor \g \rfloor - \lfloor \g/2 \rfloor - 1.
 \end{align*}
Therefore,  $ e^{-(n+1-\g)\tau}  (e^{(n+1-\g)\tau}V)^{\widetilde{}}$ is also a solution of (\ref{5.10}). By the uniqueness of the solution of (\ref{5.10}), we get (\ref{5.11}).

\vspace{0.3cm}

\textbf{Proof of Theorem \ref{thm:boundary-operator-conformal-covariance}}.
  By Lemma \ref{lm5.2}, we need only to show the first equation when  $\lfloor\gamma/2\rfloor+1\leq j\leq\lfloor\gamma\rfloor$ and the second equation when $\lfloor\gamma\rfloor-\lfloor\gamma/2\rfloor\leq j\leq\lfloor\gamma\rfloor$.

  By using (\ref{5.11}) and Lemma \ref{lm5.2},
we have, for $\lfloor\gamma/2\rfloor+1\leq j\leq\lfloor\gamma\rfloor$,
    \begin{align*}
    &e^{(-(n+1) + \g  - 2j)\tau|_{\mathbb{H}^{n}}}  B_{2j}^{2\gamma}(e^{(n+1-\gamma)\tau}V)\\
    =&\frac{1}{ b_{ 2j }}e^{(-(n+1) + \g  - 2j)\tau|_{\mathbb{H}^{n}}} \rho^{-(n+1) + \g  - 2j}  \Pi^{\g}_{j} \left(  \rho^{n+1-\g }(e^{(n+1-\gamma)\tau}V-(e^{(n+1-\g)\tau}V)^{\widetilde{}}\;)\right) |_{\rho=0}+\\
    &e^{(-(n+1) + \g  - 2j)\tau|_{\mathbb{H}^{n}}}\frac{2^{\gamma-2j}}{c_{2j-\gamma}}P_{2j-\gamma}B^{2\gamma}_{2\gamma-2j}(e^{(n+1-\gamma)\tau}V)\\
    =&\frac{1}{ b_{ 2j }}e^{(-(n+1) + \g  - 2j)\tau|_{\mathbb{H}^{n}}} \rho^{-(n+1) + \g  - 2j}  \Pi^{\g}_{j} \left(  \rho^{n+1-\g }(e^{(n+1-\gamma)\tau}V-e^{(n+1-\gamma)\tau}\widehat{\tilde{V}})\right) |_{\rho=0}+\\
    &e^{(-(n+1) + \g  - 2j)\tau|_{\mathbb{H}^{n}}}\frac{2^{\gamma-2j}}{c_{2j-\gamma}}P_{2j-\gamma}B^{2\gamma}_{2\gamma-2j}(e^{(n+1-\gamma)\tau}V)\\
     =&\frac{1}{b_{2j}}\widehat{\rho}^{\;-2j-n-1+\gamma}
      \Pi^{\g}_{j}(\widehat{\rho}^{n+1-\g}(V-\widehat{\tilde V}))\big|_{\widehat{\rho}=0}+
      \frac{2^{\gamma-2j}}{c_{2j-\gamma}}e^{(-(n+1) + \g  - 2j)\tau|_{\mathbb{H}^{n}}}     P_{2j-\gamma}\left(e^{(n+1 + \g  - 2j)\tau|_{\mathbb{H}^{n}}} \widehat{B}^{2\gamma}_{2\gamma-2j}(V)\right)\\
      =&\frac{1}{b_{2j}}\widehat{\rho}^{\;-2j-n-1+\gamma}
      \Pi^{\g}_{j}(\widehat{\rho}^{n+1-\g}(V-\widehat{\tilde V}))\big|_{\widehat{\rho}=0}+
      \frac{2^{\gamma-2j}}{c_{2j-\gamma}}\widehat{P}_{2j-\gamma}\left( \widehat{B}^{2\gamma}_{2\gamma-2j}(V)\right)\\
    =&\widehat{ B}_{2 j}^{2\g}(V).
  \end{align*}
To get the fourth  equality,  we use the fact
\begin{align*}
P_{2j-\gamma}\left(e^{(n+1 + \g  - 2j)\tau|_{\mathbb{H}^{n}}} \widehat{B}^{2\gamma}_{2\gamma-2j}(V)\right)=
e^{(n+1 - \g  + 2j)\tau|_{\mathbb{H}^{n}}} \widehat{P}_{2j-\gamma}\left( \widehat{B}^{2\gamma}_{2\gamma-2j}(V)\right)
\end{align*}
  since $P_{2j-\gamma}$ is a CR covariant operator.
This proves the first equation when $\lfloor\gamma/2\rfloor+1\leq j\leq\lfloor\gamma\rfloor$.
  The proof of the second equation when $\lfloor\gamma\rfloor-\lfloor\gamma/2\rfloor\leq j\leq\lfloor\gamma\rfloor$ is similar and we omit it.
 The proof of Theorem \ref{thm:boundary-operator-conformal-covariance} is thereby completed.

\section{Proofs of Theorems \ref{thm:dirichlet-form-symmetry} and \ref{th1.5}}
\label{sec:proofs-section}

In this section we prove Theorems  \ref{thm:dirichlet-form-symmetry} and \ref{th1.5}, as well two integral identities.
For convenience, we restate the theorems preceding their respective proofs in this section.
In order to prove the symmetry of the Dirichlet form (Theorem \ref{thm:dirichlet-form-symmetry}) and CR Sobolev trace inequalities (Theorem \ref{th1.5}), we establish a Green-type identity on $\overline{\mathcal{U}^{n+1}}$ so as to access the appropriate boundary integrals.
We also establish a highly nontrivial integral identity (Theorem \ref{thm:main-integral-identity}) relating the term
\[
  \int_{\mcU^{n+1}} U L_{2k}V \cdot \rho^{1-2[\g]} dzdt d\rho
\]
appearing in the Dirichlet form $\mathcal{Q}_{2\g}(U,V)$ to the appropriate boundary integrals.
The proof of Theorem \ref{thm:main-integral-identity} contains the majority of the work necessary for establishing the higher order trace inequalities and the symmetry $\mathcal{Q}_{2\g}(U,V) = \mathcal{Q}_{2\g}(V,U)$.

We now move to proving the symmetry of $\mathcal{Q}_{2\g}$ and the higher order Sobolev trace inequalities.
Before doing so, we establish the following Green-type identity and the integral identity in Theorem \ref{thm:main-integral-identity}.

\begin{lemma}
  Let $u,v \in \mcC^{2\g}(\mcU^{n+1}) \cap \dot H^{k,\g}(\mcU^{n+1})$.
  Then
  \[
    2^{-n}\int_{\mcU^{n+1}} u \Delta_{\B} v dV - 2^{-n}\int_{\mcU^{n+1}}v \Delta_{\B}u  dV =  \int_{\p\mcU^{n+1}} \left( \rho^{-2n-1}v \p_{\rho}u \right)|_{\rho=0} - \left( \rho^{-2n-1}u \p_{\rho} v \right) |_{\rho=0} dzdt
  \]
  \label{lem:green-identity}
\end{lemma}

\begin{proof}
  Computing
  \begin{align*}
    \int_{0}^{\oo} \rho^{-2n-3}  u \rho^{2} \p_{\rho\rho} v d\rho & = -\int_{0}^{\oo} \p_{\rho}(\rho^{-2n-1} u ) \p_{\rho} v d\rho + \left( \rho^{-2n-1}u \p_{\rho} v \right) |_{0}^{\oo}\\
    &=(2n+1)\int_{0}^{\oo} \rho^{-2n-2}u \p_{\rho} v d \rho - \int_{0}^{\oo} \rho^{-2n-1}\p_{\rho} u \cdot \p_{\rho} v d \rho  \\
    &+ \left( \rho^{-2n-1}u \p_{\rho} v \right) |_{0}^{\oo}\\
    &= -(2n+1) \int_{0}^{\oo} \p_{\rho}(\rho^{-2n-2} u) v d\rho + \int_{0}^{\oo} \p_{\rho}(\rho^{-2n-1}\p_{\rho}  u) v d\rho\\
    & + (2n+1)\left( \rho^{-2n-2} u v \right)|_{0}^{\oo} - \left( \rho^{-2n-1}\p_{\rho}u \cdot v \right)|_{0}^{\oo}\\
    &+ \left( \rho^{-2n-1}u \p_{\rho} v \right) |_{0}^{\oo}\\
    &= (2n+1)(2n+2) \int_{0}^{\oo} \rho^{-2n-3} u v d\rho - 2(2n+1) \int_{0}^{\oo} \rho^{-2n-2} \p_{\rho}u \cdot v d\rho\\
    & + \int_{0}^{\oo}\rho^{-2n-1}\p_{\rho\rho}u \cdot v d\rho + (2n+1)\left( \rho^{-2n-2} u v \right)|_{0}^{\oo} - \left( \rho^{-2n-1}\p_{\rho}u \cdot v \right)|_{0}^{\oo}\\
    &+ \left( \rho^{-2n-1}u \p_{\rho} v \right) |_{0}^{\oo}\\
  \end{align*}
  and
  \begin{align*}
    -(2n+1)\int_{0}^{\oo} \rho^{-2n-3}  u \rho \p_{\rho} v d\rho &= (2n+1) \int_{0}^{\oo} \p_{\rho} (\rho^{-2n-2}  u ) v d\rho  - (2n+1) \left( \rho^{-2n-2} u v \right)|_{0}^{\oo}\\
    &= - (2n+1)(2n+2) \int_{0}^{\oo} \rho^{-2n-3}  u v d\rho \\
    &+ (2n+1) \int_{0}^{\oo} \rho^{-2n-2} v\p_{\rho}u d\rho - (2n+1) \left( \rho^{-2n-2} u v \right)|_{0}^{\oo}
  \end{align*}
  we get
  \begin{align*}
    &\int_{0}^{\oo} \rho^{-2n-3}  u \rho^{2} \p_{\rho\rho} v d\rho -(2n+1)\int_{0}^{\oo} \rho^{-2n-3}  u \rho \p_{\rho} v d\rho \\
    &= \int_{0}^{\oo} \rho^{2} \p_{\rho\rho}u \cdot v \rho^{-2n-3} d\rho - (2n+1)\int_{0}^{\oo} \rho \p_{\rho} u \cdot v \rho^{-2n-3}d\rho\\
    & - \left( \rho^{-2n-1}\p_{\rho}u \cdot v \right)|_{0}^{\oo} + \left( \rho^{-2n-1}u \p_{\rho} v \right) |_{0}^{\oo}.
  \end{align*}
  Therefore, using the self-adjointness of $\Delta_{b}$ and $\p_{tt}$, we get the desired identity since
  \[
    \Delta_{\B} = \rho^{2}\p_{\rho\rho} +  \rho^{2} \Delta_{b} + \rho^{4} \p_{tt} - (2n+1) \rho \p_{\rho}.
  \]
\end{proof}

Now, given $U,V \in \mcC^{2\g} \cap \dot H^{k,\g}(\mcU^{n+1})$, let $\tilde U,\tilde V$ be the respective solutions of the Dirichlet problem \eqref{eq:dirichlet-problem-with-boundary-data-given-by-U}, i.e.,
\begin{equation}
  \begin{cases}
    L_{2k} \tilde{U} = 0 & \text{ in } \mcU^{n+1}\\
    B_{2j}^{2\g}\tilde (U) = B_{2j}^{2\g} (U) & 0 \leq j \leq \lfloor \g/2 \rfloor\\
    B_{2j + 2[\g]}^{2\g} \tilde (U) = B_{2j+2[\g]}^{2\g} (U) & 0 \leq j \leq \lfloor \g \rfloor - \lfloor \g/2 \rfloor - 1
  \end{cases}.
\end{equation}
and similarly for $\tilde V$.
For notational simplicity, we assume $\lfloor \g \rfloor \in 2 \N_{>0}$ so that we may freely write $\lfloor \g /2 \rfloor = \lfloor \g \rfloor /2$; the case $\lfloor \g \rfloor \in 2 \N_{>0} - 1$ is treated identically.
From Theorem \ref{thm:solution-to-l2k-dirichlet-problem} we have
\begin{equation}
  \tilde U = \sum_{j=0}^{\lfloor \g/2 \rfloor} U_{j} + \sum_{j=0}^{\lfloor \g \rfloor - \lfloor \g/2 \rfloor -1} U_{j}',
  \label{eq:expansion-for-u-tilde}
\end{equation}
where
\begin{align*}
  U_{j} &= 2^{\frac{n+1-\g+2j}{2}}\rho^{-(n+1) + \g} \mcP\left( \frac{n+1+\g-2j}{2} \right) B_{2j}^{2\g}(U)\\
  U_{j}' &= 2^{\frac{n+1- \lfloor \g \rfloor + [\g] + 2j}{2} } \rho^{-(n+1) + \g} \mcP \left( \frac{n+1+ \lfloor \g \rfloor - [\g] - 2j}{2} \right) B_{2j + 2[\g]}^{2\g} (U).
\end{align*}
For $j=0,\ldots,\lfloor \g \rfloor - \lfloor \g/2 \rfloor -1$, set $U_{\lfloor \g/2 \rfloor +j + 1} = U_{\lfloor \g \rfloor - \lfloor \g/2 \rfloor -1 -j}'$ and so
\[
  U_{\lfloor \g/2 \rfloor + j + 1} = 2^{\frac{n + 1 + \g - 2 ( \lfloor \g/2 \rfloor + j + 1)}{2}} \rho^{-(n+1) +\g} \mcP \left(  \frac{n + 1 - \g + 2 ( \lfloor \g/2 \rfloor + j + 1)}{2} \right) B_{2 \g - 2\lfloor \g/2 \rfloor - 2 -2j}^{2\g}(U).
\]
It follows that
\begin{equation}
  \left( \tilde\Delta_{\B} - (\g - 2j)^{2} \right) \rho^{n+1-\g}U_{j} = 0, \qquad j = 0, \ldots, \lfloor \g \rfloor .
  \label{eq:uj-satisfies-scattering-equation}
\end{equation}
By scattering theorem (see Section \ref{sec:dirichlet-problem-and-scattering}), we have
\begin{equation}
  U_{j}=
  \begin{cases}
    \rho^{2j}F_{j} +2^{2j-\gamma} \rho^{2\g-2j}G_{j} & j = 0 , \ldots, \lfloor \g /2 \rfloor\\
    \rho^{2\g - 2j}F_{j} + 2^{\gamma-2j}\rho^{2j}G_{j} & j= \lfloor \g /2 \rfloor + 1, \ldots, \lfloor \g \rfloor
  \end{cases},
  \label{eq:scattering-uj}
\end{equation}
where $F_{j},G_{j} \in C_{\operatorname{even}}^{\oo} ( \overline{\mcU^{n+1}})$ and
\begin{equation}
  \begin{aligned}
    F_{j}|_{\rho=0} &=
    \begin{cases}
      B_{2j}^{2\g}(U), & j=0, \ldots, \lfloor \g/2 \rfloor\\
      B_{2\g-2j}^{2\g}(U), & j= \lfloor \g/2\rfloor + 1,\ldots, \lfloor \g \rfloor
    \end{cases}\\
    G_{j}|_{\rho=0} & =
    \begin{cases}
      S\left( \frac{n+1+\g-2j}{2} \right) B_{2j}^{2\g}(U) & j=0, \ldots, \lfloor \g/2 \rfloor\\
      S \left(  \frac{n + 1 + 2 j-\gamma}{2} \right) B_{2\g-2j}^{2\g} (U) & j= \lfloor \g/2\rfloor + 1,\ldots, \lfloor \g \rfloor
    \end{cases}.
  \end{aligned}
  \label{eq:fj-gj-boundary-values}
\end{equation}
We can prepare all the same for $V$ and $\tilde V$.

We now state and prove the necessary integral identity.

\begin{theorem}\label{thm:main-integral-identity}
  Letting $U$, $\tilde U$, $V$ and $\tilde V$ be as above, there holds
  \begin{align*}
    &\int_{\mcU^{n+1}} U L_{2k}V \cdot \rho^{1-2[\g]} dzdt d\rho \\
    =&2^{-n}\int_{\mcU^{n+1}} \rho^{n+1-\g} (U - \tilde U) L_{2k}^{+}( \rho^{n+1-\g} (V-\tilde V)) dV\\
    &+\sum_{j=0}^{\lfloor \g/2 \rfloor} \sigma_{j,\g}  \int_{\p\mcU^{n+1}} B_{2j}^{2\g}(U)  B_{2\g-2j}^{2\g}(V)  dzdt + \sum_{j=\lfloor \g/2 \rfloor + 1}^{\lfloor \g \rfloor} \sigma_{j,\g} \int_{\p\mcU^{n+1}}   B_{2\g-2j}^{2\g}(U)  B_{2j}^{2\g}(V)  dzdt\\
    &+ \sum_{j=0}^{\lfloor \g/2 \rfloor} \varsigma_{j,\g} \int_{\p\mcU^{n+1}}   B_{2j}^{2\g}(U) P_{\g-2j}B_{2j}^{2\g} V    dzdt+ \sum_{j=\lfloor \g /2 \rfloor + 1}^{\lfloor \g \rfloor} \varsigma_{j,\g} \int_{\p\mcU^{n+1}}   B_{2\g - 2j}^{2\g}(U) P_{2j - \g}B_{2\g-2j}^{2\g} (V)    dzdt,
  \end{align*}
  where
  \begin{align*}
    \sigma_{j,\g} &=
    \begin{cases}
       2^{2\lfloor\g\rfloor+1}j!(\lfloor\g\rfloor-j)!\frac{\Gamma(\gamma+1-j)\Gamma(j+1-[\g])}{\Gamma(\gamma-2j)\Gamma(2j+1-\g)},& j = 0,\ldots, \lfloor \g/2 \rfloor\\
     -2^{2\lfloor\g\rfloor+1}j!(\lfloor\g\rfloor-j)!\frac{\Gamma(\gamma+1-j)\Gamma(j+1-[\g])}{\Gamma(\gamma-2j)\Gamma(2j+1-\g)}, & j = \lfloor \g /2 \rfloor + 1,\ldots, \lfloor \g \rfloor\\
    \end{cases}\\
    \varsigma_{j,\g} &=
    \begin{cases}
       2^{4j-2[\g]+1}j!(\lfloor\g\rfloor-j)!\frac{\Gamma(\gamma+1-j)\Gamma(j+1-[\gamma])}{\Gamma(\gamma+1-2j)\Gamma(\gamma-2j)}, & j = 0,\ldots, \lfloor \g /2 \rfloor\\
      2^{4\gamma-4j-2[\g]+1}j!(\lfloor\g\rfloor-j)!\frac{\Gamma(\gamma+1-j)\Gamma(j+1-[\gamma])}
      {\Gamma(2j+1-\gamma)\Gamma(2j-\gamma)}, & j = \lfloor \g/2 \rfloor + 1, \ldots, \lfloor \g \rfloor.
    \end{cases}
  \end{align*}
\end{theorem}

\begin{proof}

  Using (\ref{1.13b}) and
$
    L_{2k}^{+}(\rho^{n+1-\g}\tilde V) = 0,
$
  we get
  \begin{align*}
   \int_{\mcU^{n+1}} U L_{2k}V \cdot \rho^{1-2[\g]} dzdt d\rho &= 2^{-n}\int_{\mcU^{n+1}} \rho^{n+1-\g} U L_{2k}^{+} (\rho^{n+1-\g}V) dV\\
    &= 2^{-n}\int_{\mcU^{n+1}} \rho^{n+1-\g} (U - \tilde U) L_{2k}^{+}( \rho^{n+1-\g} (V-\tilde V)) dV\\
    &+ 2^{-n}\int_{\mcU^{n+1}} \rho^{n+1-\g} \tilde U L_{2k}^{+} (\rho^{n+1-\g}(V- \tilde V))dV\\
    &=: A_{1} + A_{2}.
  \end{align*}
  We will use the expansions \eqref{eq:expansion-for-u-tilde} for $\tilde U$ and $\tilde V$ and the Green-type identity Lemma \ref{lem:green-identity} to rewrite $A_{2}$ in terms of boundary integrals involving the boundary operators $B_{\a}^{2\g}$.
  Using \eqref{eq:uj-satisfies-scattering-equation} and the Green-type identity (Lemma \ref{lem:green-identity}), we get
  \begin{align*}
    &2^{-n}\int_{\mcU^{n+1}} \rho^{n+1 - \g} U_{j} L_{2k}^{+}(\rho^{n+1-\g}(V-\tilde V)) dV \\
    =& 2^{-n}\int_{\mcU^{n+1}} \rho^{n+1-\g} U_{j} L_{2k}^{+} (\rho^{n+1-\g}(V - \tilde V)) dV \\
    &+ 2^{-n}\int_{\mcU^{n+1}} \left( \tilde\Delta_{\B} - (\g-2j)^{2} \right)(\rho^{n+1-\g}U_{j}) \Pi^{\g}_{j}(\rho^{n+1-\g}(V-\tilde V)) dV\\
    =& -2^{-n}\int_{\mcU^{n+1}} \rho^{n+1-\g} U_{j} \Delta_{\B} \Pi^{\g}_{j}(\rho^{n+1-\g}(V-\tilde V)) dV \\
    &+ 2^{-n}\int_{\mcU^{n+1}} \Delta_{\B}(\rho^{n+1-\g}U_{j}) \Pi^{\g}_{j}(\rho^{n+1-\g}(V-\tilde V)) dV\\
    =&-\int_{\p\mcU^{n+1}}\rho^{-2n-1}\p_{\rho}(\rho^{n+1-\g}U_{j}) \Pi^{\g}_{j}(\rho^{n+1-\g}(V-\tilde V))\big|_{\rho=0} dzdt \\
    &+ \int_{\p\mcU^{n+1}} \rho^{-2n-1}\rho^{n+1-\g}U_{j} \p_{\rho}\Pi^{\g}_{j}(\rho^{n+1-\g}(V-\tilde V))\big|_{\rho=0} dzdt.
  \end{align*}
  Therefore,
  \begin{align*}
    A_{2} =& 2^{-n}\int_{\mcU^{n+1}} \rho^{n+1-\g}\tilde U L_{2k}^{+} ( \rho^{n+1 - \g} (V - \tilde V))dV\\
    =& 2^{-n}\sum_{j=0}^{\lfloor \g \rfloor} \int_{\mcU^{n+1}} \rho^{n+1 - \g} U_{j} L_{2k}^{+}(\rho^{n+1-\g}(V-\tilde V)) dV \\
    =&-\sum_{j = 0 }^{\lfloor \g \rfloor} \int_{\p\mcU^{n+1}}\rho^{-2n-1}\p_{\rho}(\rho^{n+1-\g}U_{j}) \Pi^{\g}_{j}(\rho^{n+1-\g}(V-\tilde V))\big|_{\rho=0} dzdt \\
    &+ \sum_{j=0}^{\lfloor \g \rfloor} \int_{\p\mcU^{n+1}}U_{j} \rho^{-n-\g}\p_{\rho}\Pi^{\g}_{j}(\rho^{n+1-\g}(V-\tilde V))\big|_{\rho=0} dzdt.
  \end{align*}

  Using(\ref{eq:scattering-uj})-(\ref{eq:fj-gj-boundary-values}) and (\ref{4.7})-(\ref{4.9}), we have
  \begin{align*}
&\rho^{-2n-1}\p_{\rho}(\rho^{n+1-\g}U_{j}) \Pi^{\g}_{j}(\rho^{n+1-\g}(V-\tilde V))\big|_{\rho=0}\\
=&  \begin{cases}
    (n+1-\gamma+2j)B_{2j}^{2\g}(U) \times \rho^{-n-1-\g+2j} \Pi_{j}^{\g}(\rho^{n+1-\gamma}(  V - \tilde V) )\big|_{\rho=0},  & j = 0 , \ldots, \lfloor \g /2 \rfloor\\
    n+1+\gamma-2j)B_{2\gamma-2j}^{2\g}(U) \times \rho^{-n-1+\g-2j} \Pi_{j}^{\g}(\rho^{n+1-\gamma}(  V - \tilde V) )\big|_{\rho=0}, & j= \lfloor \g /2 \rfloor + 1, \ldots, \lfloor \g \rfloor
  \end{cases}\\
=&    \begin{cases}
    (n+1-\gamma+2j)b_{2\gamma-2j}B_{2j}^{2\g}(U) (B_{2\gamma-2j}^{2\g}(V)-2^{2j-\gamma}c_{\gamma-2j}^{-1}P_{\gamma-2j}B_{2j}^{2\g}(V)), & j = 0 , \ldots, \lfloor \g /2 \rfloor\\
    n+1+\gamma-2j)b_{2j}B_{2\gamma-2j}^{2\g}(U) (B_{2j}^{2\g}(V)-2^{\gamma-2j}c_{2j-\gamma}^{-1}P_{2j-\gamma}B_{2\gamma-2j}^{2\g}(V)), & j= \lfloor \g /2 \rfloor + 1, \ldots, \lfloor \g \rfloor.
  \end{cases}
  \end{align*}
Similarly,
    \begin{align*}
&U_{j} \rho^{-n-\g}\p_{\rho}\Pi^{\g}_{j}(\rho^{n+1-\g}(V-\tilde V))\big|_{\rho=0}\\
=&  \begin{cases}
    (n+1+\gamma-2j)B_{2j}^{2\g}(U) \times \rho^{-n-1-\g+2j} \Pi_{j}^{\g}(\rho^{n+1-\gamma}(  V - \tilde V) )\big|_{\rho=0},  & j = 0 , \ldots, \lfloor \g /2 \rfloor\\
    n+1-\gamma+2j)B_{2\gamma-2j}^{2\g}(U) \times \rho^{-n-1+\g-2j} \Pi_{j}^{\g}(\rho^{n+1-\gamma}(  V - \tilde V) )\big|_{\rho=0}, & j= \lfloor \g /2 \rfloor + 1, \ldots, \lfloor \g \rfloor
  \end{cases}\\
=&    \begin{cases}
    (n+1+\gamma-2j)b_{2\gamma-2j}B_{2j}^{2\g}(U) (B_{2\gamma-2j}^{2\g}(V)-2^{2j-\gamma}c_{\gamma-2j}^{-1}P_{\gamma-2j}B_{2j}^{2\g}(V)), & j = 0 , \ldots, \lfloor \g /2 \rfloor\\
    n+1-\gamma+2j)b_{2j}B_{2\gamma-2j}^{2\g}(U) ( B_{2j}^{2\g}(V)-2^{\gamma-2j}c_{2j-\gamma}^{-1}P_{2j-\gamma}B_{2\gamma-2j}^{2\g}(V)), & j= \lfloor \g /2 \rfloor + 1, \ldots, \lfloor \g \rfloor.
  \end{cases}
  \end{align*}
 Thus, we obtain
  \begin{align*}
    A_2  =&\sum_{j=0}^{\lfloor \g/2 \rfloor}2(\gamma-2j) b_{2\gamma-2j}\int_{\p\mcU^{n+1}}B_{2j}^{2\g}(U) (B_{2\gamma-2j}^{2\g}(V)-2^{2j-\gamma}c_{\gamma-2j}^{-1}P_{\gamma-2j}B_{2j}^{2\g}(V))dzdt\\
        &+\sum_{j=\lfloor \g/2 \rfloor+1}^{\lfloor \g \rfloor} 2(2j-\gamma)b_{2j}\int_{\p\mcU^{n+1}} B_{2\gamma-2j}^{2\g}(U) (B_{2j}^{2\g}(V)-2^{\gamma-2j}c_{2j-\gamma}^{-1}P_{2j-\gamma}B_{2\gamma-2j}^{2\g}(V))dzdt.
  \end{align*}
The desired result follows by setting
  \begin{align*}
    \sigma_{j,\g} &=
    \begin{cases}
      2(\gamma-2j) b_{2\gamma-2j}= 2^{2\lfloor\g\rfloor+1}j!(\lfloor\g\rfloor-j)!\frac{\Gamma(\gamma+1-j)\Gamma(j+1-[\g])}{\Gamma(\gamma-2j)\Gamma(2j+1-\g)},& j = 0,\ldots, \lfloor \g/2 \rfloor\\
    2(2j-\gamma)b_{2j}= -2^{2\lfloor\g\rfloor+1}j!(\lfloor\g\rfloor-j)!\frac{\Gamma(\gamma+1-j)\Gamma(j+1-[\g])}{\Gamma(\gamma-2j)\Gamma(2j+1-\g)}, & j = \lfloor \g /2 \rfloor + 1,\ldots, \lfloor \g \rfloor\\
    \end{cases}\\
    \varsigma_{j,\g} &=
    \begin{cases}
       -2^{2j-\gamma}c_{\gamma-2j}^{-1}\sigma_{j,\g}=2^{4j-2[\g]+1}j!(\lfloor\g\rfloor-j)!\frac{\Gamma(\gamma+1-j)\Gamma(j+1-[\gamma])}{\Gamma(\gamma+1-2j)\Gamma(\gamma-2j)}, & j = 0,\ldots, \lfloor \g /2 \rfloor\\
      -2^{\gamma-2j}c_{2j-\gamma}^{-1}\sigma_{j,\g}=2^{4\gamma-4j-2[\g]+1}j!(\lfloor\g\rfloor-j)!\frac{\Gamma(\gamma+1-j)\Gamma(j+1-[\gamma])}
      {\Gamma(2j+1-\gamma)\Gamma(2j-\gamma)}, & j = \lfloor \g/2 \rfloor + 1, \ldots, \lfloor \g \rfloor.
    \end{cases}
  \end{align*}
These complete the proof of Theorem \ref{thm:main-integral-identity}.
\end{proof}

\textbf{Proof of Theorem \ref{thm:dirichlet-form-symmetry}}
  We aim to show that
  \begin{align*}
    \mathcal{Q}_{2\gamma}(U,V):=&\int_{\mathcal{U}^{n+1}}U L_{2k} V \rho^{1-2[\gamma]}dzdtd\rho \\
    &- \sum_{j=0}^{\lfloor \g/2 \rfloor} \sigma_{j,\g}  \int_{\p\mcU^{n+1}} B_{2j}^{2\g}(U)  B_{2\g-2j}^{2\g}(V)  dzdt - \sum_{j=\lfloor \g/2 \rfloor + 1}^{\lfloor \g \rfloor} \sigma_{j,\g} \int_{\p\mcU^{n+1}}   B_{2\g-2j}^{2\g}(U)  B_{2j}^{2\g}(V)  dzdt\\
    =&\int_{\mathcal{U}^{n+1}}V L_{2k} U \rho^{1-2[\gamma]}dzdtd\rho \\
    &- \sum_{j=0}^{\lfloor \g/2 \rfloor} \sigma_{j,\g}  \int_{\p\mcU^{n+1}} B_{2j}^{2\g}(V)  B_{2\g-2j}^{2\g}(U)  dzdt - \sum_{j=\lfloor \g/2 \rfloor + 1}^{\lfloor \g \rfloor} \sigma_{j,\g} \int_{\p\mcU^{n+1}}   B_{2\g-2j}^{2\g}(V)  B_{2j}^{2\g}(U)  dzdt\\
    =&\mathcal{Q}_{2\g}(V,U).
  \end{align*}
  Let $\tilde U$ and $\tilde V$ be as above.
  By Theorem \ref{thm:main-integral-identity}, there holds
  \begin{align*}
    \int_{\mcU^{n+1}} U L_{2k}V \cdot \rho^{1- 2[\g]} dzdt d\rho  &= 2^{-n}\int_{\mcU^{n+1}} \rho^{n+1-\g} (U - \tilde U) L_{2k}^{+}( \rho^{n+1-\g} (V-\tilde V)) dV + A_{2}(U,V)\\
    &=: A_{1}(U,V) + A_{2}(U,V),
  \end{align*}
  where
  \begin{align*}
    A_{2}(U,V) :=& \sum_{j=0}^{\lfloor \g/2 \rfloor} \sigma_{j,\g}  \int_{\p\mcU^{n+1}} B_{2j}^{2\g}(U)  B_{2\g-2j}^{2\g}(V)  dzdt+\sum_{j=\lfloor \g/2 \rfloor + 1}^{\lfloor \g \rfloor} \sigma_{j,\g} \int_{\p\mcU^{n+1}}   B_{2\g-2j}^{2\g}(U)  B_{2j}^{2\g}(V)  dzdt\\
    +& \sum_{j=0}^{\lfloor \g/2 \rfloor} \varsigma_{j,\g} \int_{\p\mcU^{n+1}}   B_{2j}^{2\g}(U) P_{\g-2j}B_{2j}^{2\g} (V)    dzdt + \sum_{j=\lfloor \g /2 \rfloor + 1}^{\lfloor \g \rfloor} \varsigma_{j,\g} \int_{\p\mcU^{n+1}}   B_{2\g - 2j}^{2\g}(U) P_{2j - \g}B_{2\g-2j}^{2\g} (V)    dzdt.
  \end{align*}
  Since $\rho^{n+1-\g}(U - \tilde U)$ and $\rho^{n+1-\g}(V-\tilde V)$ decay sufficiently fast at $\rho=0$, we may use the self-adjointness of $\Delta_{\mathbb{B}}$ to conclude
  \[
    A_{1}(U,V) = A_{1}(V,U).
  \]
  To show $A_{2}(U,V) = A_{2}(V,U)$, we use the self-adjointness of $P_{\g-2j}$ for $j = 0,\ldots, \lfloor \g /2 \rfloor$ and $P_{2j - \g}$ for $j = \lfloor \g/2 \rfloor + 1, \ldots, \lfloor \g \rfloor$ to obtain
  \begin{align*}
    A_{2}(U,V) &= \sum_{j=0}^{\lfloor \g/2 \rfloor} \sigma_{j,\g}  \int_{\p\mcU^{n+1}} B_{2j}^{2\g}(U)  B_{2\g-2j}^{2\g}(V)  dzdt+\sum_{j=\lfloor \g/2 \rfloor + 1}^{\lfloor \g \rfloor} \sigma_{j,\g} \int_{\p\mcU^{n+1}}   B_{2\g-2j}^{2\g}(U)  B_{2j}^{2\g}(V)  dzdt\\
    &+ \sum_{j=0}^{\lfloor \g/2 \rfloor} \varsigma_{j,\g} \int_{\p\mcU^{n+1}}   B_{2j}^{2\g}(V) P_{\g-2j}B_{2j}^{2\g} (U)    dzdt + \sum_{j=\lfloor \g /2 \rfloor + 1}^{\lfloor \g \rfloor} \varsigma_{j,\g} \int_{\p\mcU^{n+1}}   B_{2\g - 2j}^{2\g}(V) P_{2j - \g}B_{2\g-2j}^{2\g} (U)    dzdt.
  \end{align*}
  Using these observations, we get
  \begin{align*}
    \mathcal{Q}_{2\gamma}(U,V):=&\int_{\mathcal{U}^{n+1}}U L_{2k} V \rho^{1-2[\gamma]}dzdtd\rho \\
    &- \sum_{j=0}^{\lfloor \g/2 \rfloor} \sigma_{j,\g}  \int_{\p\mcU^{n+1}} B_{2j}^{2\g}(U)  B_{2\g-2j}^{2\g}(V)  dzdt-\sum_{j=\lfloor \g/2 \rfloor + 1}^{\lfloor \g \rfloor} \sigma_{j,\g} \int_{\p\mcU^{n+1}}   B_{2\g-2j}^{2\g}(U)  B_{2j}^{2\g}(V)  dzdt\\
    =& A_{1}(U,V) + A_{2}(U,V)\\
    &- \sum_{j=0}^{\lfloor \g/2 \rfloor} \sigma_{j,\g}  \int_{\p\mcU^{n+1}} B_{2j}^{2\g}(U)  B_{2\g-2j}^{2\g}(V)  dzdt-\sum_{j=\lfloor \g/2 \rfloor + 1}^{\lfloor \g \rfloor} \sigma_{j,\g} \int_{\p\mcU^{n+1}}   B_{2\g-2j}^{2\g}(U)  B_{2j}^{2\g}(V)  dzdt\\
    =& A_{1}(U,V)+ \\
    &\sum_{j=0}^{\lfloor \g/2 \rfloor} \varsigma_{j,\g} \int_{\p\mcU^{n+1}}   B_{2j}^{2\g}(V) P_{\g-2j}B_{2j}^{2\g} (U)    dzdt +
    \sum_{j=\lfloor \g /2 \rfloor + 1}^{\lfloor \g \rfloor} \varsigma_{j,\g} \int_{\p\mcU^{n+1}}   B_{2\g - 2j}^{2\g}(V) P_{2j - \g}B_{2\g-2j}^{2\g} (U)    dzdt\\
        =& A_{1}(V,U)+ \\
    &\sum_{j=0}^{\lfloor \g/2 \rfloor} \varsigma_{j,\g} \int_{\p\mcU^{n+1}}   B_{2j}^{2\g}(U) P_{\g-2j}B_{2j}^{2\g} (V)    dzdt +
    \sum_{j=\lfloor \g /2 \rfloor + 1}^{\lfloor \g \rfloor} \varsigma_{j,\g} \int_{\p\mcU^{n+1}}   B_{2\g - 2j}^{2\g}(U) P_{2j - \g}B_{2\g-2j}^{2\g} (V)    dzdt\\
    =& \mathcal{Q}_{2\g}(V,U),
  \end{align*}
  which is what we wanted to show. These complete the proof of  Theorem \ref{thm:dirichlet-form-symmetry}.
\vspace{0.3cm}

At last, we can now prove the CR trace inequalities recorded in Theorem \ref{th1.5}, which we recorded again below for convenience.
\begin{theorem}
  Let  $\gamma\in (0,n+1)\setminus\mathbb{N}$ and $U\in \mathcal{C}^{2\gamma}(\mathcal{U}^{n+1})\cap\dot{H}^{k,[\gamma]}(\mathcal{U}^{n+1})$. Set $f^{(2j)}=B^{2\gamma}_{2j}(U)$, $\phi^{(2j)}=B^{2\gamma}_{2j+2[\gamma]}(U)$ and
  $\mathcal{E}_{2\gamma}(U)=\mathcal{Q}_{2\gamma}(U,U)$. It holds that
  \begin{align*}
    \mathcal{E}_{2\g}(U) \geq \sum_{j=0}^{\lfloor \g/2 \rfloor} \varsigma_{j,\g} \int_{\p\mcU^{n+1}}   B_{2j}^{2\g}(U) P_{\g-2j}B_{2j}^{2\g} (U)    dzdt + \sum_{j=\lfloor \g /2 \rfloor + 1}^{\lfloor \g \rfloor} \varsigma_{j,\g} \int_{\p\mcU^{n+1}}   B_{2\g - 2j}^{2\g}(U) P_{2j - \g}B_{2\g-2j}^{2\g} (U)    dzdt.
  \end{align*}
    Equality is attained if and only if $U$ is the unique solution of the Dirichlet boundary value
  problem
  \begin{equation*}
    \begin{cases}
      L_{2k} U = 0 & \text{in }\mcU^{n+1}\\
      B_{2j}^{2\g}(U) = f^{(2j)},& 0 \leq j \leq [\g/2]\\
      B_{2j+2[\g]}^{2\g}(U) = \phi^{(2j)}, & 0 \leq j \leq [\g] - [\g/2] -1
    \end{cases}.
  \end{equation*}
\end{theorem}

\begin{proof}
  Let $\tilde U$ be as above.
  By Theorem \ref{thm:main-integral-identity}, there holds
  \begin{align*}
    \mathcal{E}_{2\g}(U) :=& \mathcal{Q}_{2\gamma}(U,U)\\
    =&\int_{\mathcal{U}^{n+1}}U L_{2k} U \rho^{1-2[\gamma]}dzdtd\rho \\
    &- \sum_{j=0}^{\lfloor \g/2 \rfloor} \sigma_{j,\g}  \int_{\p\mcU^{n+1}} B_{2j}^{2\g}(U)  B_{2\g-2j}^{2\g}(U)  dzdt+\sum_{j=\lfloor \g/2 \rfloor + 1}^{\lfloor \g \rfloor} \sigma_{j,\g} \int_{\p\mcU^{n+1}}   B_{2\g-2j}^{2\g}(U)  B_{2j}^{2\g}(U)  dzdt\\
    =& A_{1}(U,U)\\
    &+ \sum_{j=0}^{\lfloor \g/2 \rfloor} \varsigma_{j,\g} \int_{\p\mcU^{n+1}}   B_{2j}^{2\g}(U) P_{\g-2j}B_{2j}^{2\g} (U)    dzdt + \sum_{j=\lfloor \g /2 \rfloor + 1}^{\lfloor \g \rfloor} \varsigma_{j,\g} \int_{\p\mcU^{n+1}}   B_{2\g - 2j}^{2\g}(U) P_{2j - \g}B_{2\g-2j}^{2\g} (U)    dzdt.
  \end{align*}
  Lastly, it is easy to see $A_{1}(U,U) \geq 0 $ since the bottom of the spectrum of $-\Delta_{\B}$ is $(n+1)^{2}$.
  Moreover, $A_{1}(U,U) = 0$ if and only if $U = \tilde U$.
  This is enough to conclude the result.
\end{proof}

\section{Boundary Operators and Sobolev trace inequalities on complex hyperbolic ball: Proofs of Theorems \ref{th1.9},   \ref{th1.10} and \ref{th1.11}}.
\label{sec:scattering-problem-solution}

In a similar way, we define the boundary operators on complex hyperbolic ball $\B_{\C}^{n+1}=\{w\in\mathbb{C}^{n+1}: |w|<1\}$ as follows.
For $g\in \mathcal{C}^{2\gamma}(\mathbb{B}_{\mathbb{C}}^{n+1})$ and $j\in \mathbb{N}_{>0}$,
where
\begin{align}
  \mathcal{ C}^{2\gamma}(\mathbb{B}_{\mathbb{C}}^{n+1})=C^{\infty}(\overline{\mathbb{B}_{\mathbb{C}}^{n+1}})+(1-|w|^{2})^{[\gamma]}
  C^{\infty}(\overline{\mathbb{B}_{\mathbb{C}}^{n+1}}),
\end{align}
we define the boundary operators on $\mathbb{B}_{\mathbb{C}}^{n+1}$ as:
\begin{itemize}
  \item $ B_{ 0,\mathbb{B}}^{2\g}(g)=g|_{|w|=1}$;
  \item  $1\leq j\leq \lfloor\gamma/2\rfloor$,
    \begin{align*}
      B_{2 j,\mathbb{B}}^{2\g}(g)=&\frac{1}{2^{j} b_{ 2j }}(1-|w|^{2})^{\frac{-(n+1) + \g - 2j}{2}}   \prod_{\ell=0}^{j-1}D_{s-l} \prod_{\ell=\lfloor\gamma\rfloor-j+1}^{\lfloor\gamma\rfloor}D_{s-l} ((1-|w|^{2})^{\frac{n+1-\g}{2} } g)|_{|w|=1};
    \end{align*}
  \item $0\leq j\leq \lfloor\gamma\rfloor-\lfloor\gamma/2\rfloor-1$,
    \begin{align*}
      B_{ 2j + 2[\g],\mathbb{B}}^{2\g }(g)=&\frac{1}{ 2^{j+[\g]}b_{ 2j + 2[\g]}}(1-|w|^{2})^{\frac{-(n+1) + \g - 2j  - 2[\g]}{2}} \prod_{\ell=0}^{j}D_{s-l} \prod_{\ell=\lfloor\gamma\rfloor-j+1}^{\lfloor\gamma\rfloor}D_{s-l} ( (1-|w|^{2})^{\frac{n+1-\g}{2}}g)  |_{|w|=1};
    \end{align*}
  \item   $\lfloor\gamma/2\rfloor+1\leq j\leq\lfloor\gamma\rfloor$,
    \begin{align*}\nonumber 
      B_{ 2j, \mathbb{B}}^{2\g }(g)
      =&\frac{1}{2^{j}b_{2j}}(1-|w|^{2})^{\frac{-2j-n-1+\gamma}{2}}
      \Pi^{\g}_{j}((1-|w|^{2})^{\frac{n+1-\g}{2}}(g-\tilde g))\big|_{|w|=1}\\
      &+\frac{2^{2j-\gamma}}{c_{2j-\gamma}}P_{2j-\gamma}^{\mathbb{S}^{2n+1}}B^{2\gamma}_{2\gamma-2j,\mathbb{B}}(g);
    \end{align*}
  \item   $\lfloor\gamma\rfloor-\lfloor\gamma/2\rfloor\leq j\leq\lfloor\gamma\rfloor$,
    \begin{align*}
      B_{ 2j+2[\gamma],\mathbb{B}}^{2\g }(g)=&\frac{1}{2^{j+[\g]}b_{2j+2[\gamma]}}(1-|w|^{2})^{\frac{-2j-2[\g]-n-1+\gamma}{2}}
      \Pi^{\g}_{\lfloor\gamma\rfloor-j}((1-|w|^{2})^{\frac{n+1-\g}{2}}(g-\tilde g))\big|_{|w|=1} \\
      &+\frac{2^{2j+2[\gamma]-\gamma}}{c_{2j+2[\gamma]-\gamma}}P^{\mathbb{S}^{2n+1}}_{2j+[\gamma]-\lfloor\gamma\rfloor}
      B^{2\gamma}_{2\lfloor\gamma\rfloor-2j,\mathbb{B}}(g);
    \end{align*}
\end{itemize}
where $\tilde{g}$ is the solution of
\begin{align*}
    \begin{cases}
      L_{2k,\mathbb{B}}\tilde{g}=0, & \text{in }\B_{\C}^{n+1}\\
       B_{2 j,\mathbb{B}}^{2\g}(\tilde{g}) =  B_{2 j,\mathbb{B}}^{2\g}(g),& 0 \leq j \leq \lfloor \g/2 \rfloor\\
      B_{ 2j + 2[\g],\mathbb{B}}^{2\g }(\tilde{g})= B_{ 2j + 2[\g],\mathbb{B}}^{2\g }(g), & 0 \leq j \leq \lfloor \g \rfloor - \lfloor \g/2 \rfloor -1.
    \end{cases}.
\end{align*}
Now we recall that the Cayley transform $\mathcal{C}: \mathcal{U}^{n+1} \rightarrow \mathbb{B}_{\mathbb{C}}^{n+1}$:
\begin{align*}
  \mathcal{C}(z)= \left(\frac{2iz_{1}}{i+z_{n+1}}, \cdots, \frac{2iz_{n}}{i+z_{n+1}},\frac{i-z_{n+1}}{i+z_{n+1}}\right):=w.
\end{align*}
We compute
\begin{align}\nonumber
 1-|w|^{2}=&1-\left(\sum_{j=1}^{n}\frac{4|z_{1}|^{2}}{|i+z_{n+1}|^{2}}+\frac{|i-z_{n+1}|^{2}}{|i+z_{n+1}|^{2}}\right)\\
 \nonumber
 =&4\frac{\textrm{Im} z_{n+1}-\sum_{j=1}^{n}|z_i|^{2}}{|i+z_{n+1}|^{2}}\\
 \label{7.2}
 =&4\frac{q}{|i+z_{n+1}|^{2}}.
\end{align}
On the other hand,
since $\mathcal{C}$ is an isometry, we have, via volume form on two models,
\begin{align*}
\frac{1}{4q^{n+2}}dxdydtdq=&\frac{dw}{(1-|w|^{2})^{n+2}};\\
|J_{\mathcal{C}}|dxdydtdq=&dw,
\end{align*}
where $J_{\mathcal{C}}$ is the Jacobian determinant.
Therefore,
\begin{align}
|J_{\mathcal{C}}|=\frac{(1-|w|^{2})^{n+2}}{4q^{n+2}}=\frac{4^{n+2}}{4|i+z_{n+1}|^{2n+4}}. \label{7.3}
\end{align}
Noticing that $z_{n+1}=t+i \textrm{Im}z_{n+1}$, we have
\begin{align*}
|i+z_{n+1}|^{2}=t^{2}+(1+\textrm{Im}z_{n+1})^{2}=t^{2}+(1+q+\sum_{j=1}^{n}|z_{j}|^{2})^{2}
\end{align*}
and thus
\begin{align}
|J_{\mathcal{C}}|=\frac{4^{n+1}}{(t^{2}+(1+q+\sum_{j=1}^{n}|z_{j}|^{2})^{2})^{n+2}}. \label{7.4}
\end{align}
By using  (\ref{7.2})-(\ref{7.4}), we have
\begin{align}\label{7.5}
1-|w|^{2}=&q|4J_{\mathcal{C}}|^{\frac{1}{n+2}}=\frac{1}{2}\rho^{2}|4J_{\mathcal{C}}|^{\frac{1}{n+2}};\\
\label{7.6}
|4J_{\mathcal{C}}|\big|_{q=0}=&|2J_{\partial\mathcal{C}}|^{\frac{n+2}{n+1}}.
\end{align}
To get (\ref{7.6}),  we use the fact  (see e.g. \cite{MR2925386,Branson2})
\begin{align}\label{7.7}
|J_{\partial\mathcal{C}}|=\frac{2^{2n+1}}{(t^{2}+(1+\sum_{j=1}^{n}|z_{j}|^{2})^{2})^{n+1}}.
\end{align}

\textbf{Proof of Theorem \ref{th1.9}}. By using (\ref{7.5})-(\ref{7.6}) and $q=\frac{1}{2}\rho^{2}$, we obtain, for $1\leq j\leq \lfloor\gamma/2\rfloor$,
    \begin{align*}
      B_{2 j,\mathbb{B}}^{2\g}(g)=&\frac{1}{ 2^{j}b_{ 2j }}(q|4J_{\mathcal{C}}|^{\frac{1}{n+2}})^{\frac{-(n+1) + \g - 2j}{2}}   \prod_{\ell=0}^{j-1}D_{s-l} \prod_{\ell=\lfloor\gamma\rfloor-j+1}^{\lfloor\gamma\rfloor}D_{s-l} ((q|4J_{\mathcal{C}}|^{\frac{1}{n+2}})^{\frac{n+1-\g}{2} } g(\mathcal{C}(z)))|_{q=0}\\
      =&\frac{1}{ b_{ 2j }}|2J_{\partial\mathcal{C}}|^{\frac{-(n+1) + \g  - 2j}{2n+2}}
      \rho^{-(n+1) + \g - 2j}   \prod_{\ell=0}^{j-1}D_{s-l} \prod_{\ell=\lfloor\gamma\rfloor-j+1}^{\lfloor\gamma\rfloor}D_{s-l} (|4J_{\mathcal{C}}|^{\frac{n+1-\g}{2(n+2)}}\rho^{n+1-\g} g(\mathcal{C}(z)))|_{\rho=0}\\
      =&|2J_{\partial\mathcal{C}}|^{\frac{-(n+1) + \g  - 2j}{2n+2}}
B_{2 j}^{2\g}\left(|4J_{\mathcal{C}}|^{\frac{n+1-\gamma}{2(n+2)}} g(\mathcal{C}(z))\right).
    \end{align*}
The rest of the proof is similar and we omit it.
\vspace{0.3cm}

\textbf{Proof of Theorem \ref{th1.10}}.
Using Theorems \ref{thm:boundary-to-fractional-operator}, \ref{th1.9}  and (\ref{1.29}),
we have, for $1\leq j\leq\lfloor\gamma/2\rfloor$,
\begin{align*}
B^{2\gamma}_{2\gamma-2j,\B}(u)=&|2J_{\partial\mathcal{C}}|^{\frac{-(n+1) - \g  + 2j}{2n+2}}
B_{2\gamma-2j}^{2\g}\left(|4J_{\mathcal{C}}|^{\frac{n+1-\gamma}{2(n+2)}} g(\mathcal{C}(z))\right)\\
=&\frac{2^{\gamma-2j}}{c_{2j-\gamma}} |2J_{\partial\mathcal{C}}|^{\frac{-(n+1) - \g  + 2j}{2n+2}} P_{\gamma-2j}B^{2\gamma}_{2j}\left(|4J_{\mathcal{C}}|^{\frac{n+1-\gamma}{2(n+2)}} u(\mathcal{C}(z))\right)\\
=&\frac{2^{\gamma-2j}}{c_{2j-\gamma}} |2J_{\partial\mathcal{C}}|^{\frac{-(n+1) - \g  + 2j}{2n+2}} P_{\gamma-2j}\left(|2J_{\partial\mathcal{C}}|^{\frac{n+1- \g+ 2j}{2n+2}}
B_{2 j, \mathbb{B}}^{2\g}(u)\right)\\
=&\frac{2^{\gamma-2j}}{c_{2j-\gamma}} P_{\gamma-2j}^{\S^{2n+1}}B^{2\gamma}_{2j,\B}(u).
\end{align*}
This proves (\ref{1.30}). The proof of (\ref{1.31}) is similar and we omit it.
These complete the proof of Theorem \ref{th1.10}.
\vspace{0.3cm}

\textbf{Proof of Theorem \ref{th1.11}}. By (\ref{7.5}), we have
\begin{align*}
&\int_{\mathbb{B}_{\mathbb{C}}^{n+1}}u L_{2k,\mathbb{B}} v (1-|w|^{2})^{-[\gamma]}dw\\
=&\int_{\mathbb{B}_{\mathbb{C}}^{n+1}}(1-|w|^{2})^{\frac{n+2-k-[\gamma]}{2} } u L_{2k,\mathbb{B}}^{+} \left( (1-|w|^{2})^{\frac{n+2-[\gamma]-k}{2}}v\right)dV\\
 =&2^{\gamma-n-1}\int_{\mcU^{n+1}} \rho^{n+1-\g} (4J_{\mathcal{C}})^{\frac{n+1-\gamma}{2(n+2)}}u(\mathcal{C}(z)) L_{2k}^{+} (\rho^{n+1-\g}(4J_{\mathcal{C}})^{\frac{n+1-\gamma}{2(n+2)}}v(\mathcal{C}(z))) dV\\
 =&2^{\gamma-1}\int_{\mcU^{n+1}}  (4J_{\mathcal{C}})^{\frac{n+1-\gamma}{2(n+2)}}u(\mathcal{C}(z)) L_{2k} ((4J_{\mathcal{C}})^{\frac{n+1-\gamma}{2(n+2)}}v(\mathcal{C}(z)))\rho^{1-2[\g]} dzdtd\rho.
\end{align*}
Similarly, by Theorem \ref{th1.9},
\begin{align*}
&\int_{\mathbb{S}^{2n+1}}B_{2j,\mathbb{B}}^{2\g}(u)B^{2\gamma}_{2\gamma-2j,\mathbb{B}}(v)d\sigma\\
=&
\int_{\mathbb{H}^{n}}|2J_{\partial\mathcal{C}}|^{-1}B_{2j}^{2\g}\left(|4J_{\mathcal{C}}|^{\frac{n+1-\gamma}{2(n+2)}} u(\mathcal{C}(z))\right) B_{2\gamma-2j}^{2\g}\left(|4J_{\mathcal{C}}|^{\frac{n+1-\gamma}{2(n+2)}} v(\mathcal{C}(z))\right) |2J_{\partial\mathcal{C}}| dzdt\\
=&\int_{\mathbb{H}^{n}}B_{2j}^{2\g}\left(|4J_{\mathcal{C}}|^{\frac{n+1-\gamma}{2(n+2)}} u(\mathcal{C}(z))\right) B_{2\gamma-2j}^{2\g}\left(|4J_{\mathcal{C}}|^{\frac{n+1-\gamma}{2(n+2)}} v(\mathcal{C}(z))\right) dzdt.
\end{align*}
To get the first equality, we use the fact (see \cite{Branson2})
\begin{align*}
\int_{\mathbb{S}^{2n+1}}fd\sigma=\int_{\mathbb{H}^{n}}f(\partial\mathcal{C}(z))|2J_{\partial\mathcal{C}}|dzdt,\;\; f\in L^{1}(\mathbb{S}^{2n+1}).
\end{align*}
Therefore, we obtain
\begin{align}\nonumber
  \mathcal{Q}_{2\gamma,\mathbb{B}}(u,v):=&2^{1-\gamma}\int_{\mathbb{B}_{\mathbb{C}}^{n+1}}u L_{2k,\mathbb{B}} v (1-|z|^{2})^{-[\gamma]}dz\\
  \nonumber
  &-\sum_{j=0}^{\lfloor\gamma/2\rfloor} \sigma_{j,\g} \int_{\mathbb{S}^{2n+1}}B_{2j,\mathbb{B}}^{2\g}(u)B^{2\gamma}_{2\gamma-2j,\mathbb{B}}(v)d\sigma -\sum_{j=\lfloor \g /2 \rfloor + 1 }^{\lfloor \g \rfloor} \sigma_{j,\g}
\int_{\mathbb{S}^{2n+1}}B_{2\g-2j,\mathbb{B}}^{2\g}(u)B^{2\gamma}_{2j}(v)d\sigma\\
\label{7.8}
=&\mathcal{Q}_{2\gamma}\left(|4J_{\mathcal{C}}|^{\frac{n+1-\gamma}{2(n+2)}} u(\mathcal{C}(z)),|4J_{\mathcal{C}}|^{\frac{n+1-\gamma}{2(n+2)}} v(\mathcal{C}(z))\right).
\end{align}
The desired result follows by combing (\ref{7.8}) and  Theorem \ref{th1.5}.

\section{Solution to the Scattering Problem: Proof of Theorem \ref{thm:scattering-solution}}
\label{sec:scattering-problem-solution2}

In this last section, we solve explicitly the scattering problem on the complex hyperbolic ball.
More precisely, we obtain an integral representation and an expansion in terms of special functions for the solution $u$ to the scattering problem
\begin{equation*}
  \begin{cases}
    \Delta_{\varphi} u- s(n+1-s)u=0, \quad \text{ in }\B_{\C}^{n+1}\\
    u=\varphi^{n+1-s}F+\varphi^{s}G\\
    F|_{M}=f
  \end{cases},
\end{equation*}
where $\Delta_{\varphi} = \frac{1}{4}\Delta_{\B}$, $\g\in(0,\frac{n+1}{2})$, $s=\frac{n+1}{2}+\g$ and $\varphi = 1-|w|^{2}$.

It will be necessary to compute $\Delta_{\varphi}r^{j+k}Y_{j,k}$ for $Y_{j,k}\in\mcH_{j,k}$.
On one hand, if $F(w,\bar{w})=w_{1}^{j}\bar{w}_{2}^{k}$, then $f=F|_{\S^{2n+1}}\in\mcH_{j,k}$, and so, since $\mcH_{j,k}$ is $U(n+1)$-irreducible, every element in $\mcH_{j,k}$ is a linear combination of $U(n+1)$-translates of $f$.
On the other hand, $\Delta_{\varphi}$ is $U(n+1)$-invariant.
Therefore, to compute $\Delta_{\varphi}r^{j+k}Y_{j,k}$, it is sufficient to compute $\Delta_{\varphi}w_{1}^{j}\bar{w}_{2}^{k}$ and use linearity and $U(n+1)$-translation.
It is easy to compute
\[
  \Delta_{\varphi}w_{1}^{j}\bar{w}_{2}^{k}=jk(1-|w|^{2})w_{1}^{j}\bar{w}_{2}^{k},
\]
and so, if $\mathcal{S}\in{U(n+1)}$ is such that $\mathcal{S}(w_{1}^{j}\bar{w}_{2}^{k})|_{\S^{2n+1}}=:G\in\mcH_{j,k}$ (here $U(n+1)$ acts on functions by the regular left-action), then
\[
  \Delta_{\varphi}(r^{j+k}G)=\Delta_{\varphi}\left( \mathcal{S}(w_{1}^{j}\bar{w}_{2}^{k}) \right)=\mathcal{S}\left( \Delta_{\varphi}w_{1}^{2}\bar{w}_{2}^{k} \right)=\mathcal{S}\left( jk(1-|w|^{2})w_{1}^{j}\bar{w}_{2}^{k} \right)=jk(1-r^{2})r^{j+k}G.
\]
Therefore, by linearity of $\Delta_{\varphi}$,
\[
  \Delta_{\varphi}\left( r^{j+k}Y_{j,k} \right)=jk(1-r^{2})r^{j+k}Y_{j,k},\qquad{Y_{j,k}\in\mcH_{j,k}}.
\]
Generally, it will be useful to know that, if $u(w,\bar{w})=g(r^{2})r^{j+k}Y_{j,k}$, then (see \cite{MR0370044} or compute directly)
\begin{equation}
    \Delta_{\varphi}u=-\tilde g(r^{2})(1-r^{2})r^{j+k} Y_{j,k}
  \label{eq:laplace-on-g(r)rY}
\end{equation}
where
\[
    \tilde g(r^{2}) := \left[ r^{2}(1-r^{2})g''(r^{2})   +(j+k+n+1-(j+k+1)r^{2})g'(r^{2})-jkg(r^{2}) \right].
\]
Theorem \ref{thm:scattering-solution} may now be proved.

\begin{proof}[Proof of Theorem \ref{thm:scattering-solution}]

  Setting
  \begin{equation}
    W(w)=\frac{\Gamma^{2}(s)}{n!\Gamma(2\g)}\varphi^{2\g}\int_{\S^{2n+1}}|1-w\cdot\bar\xi|^{-2s}f(\xi)d\sigma,
    \label{eq:v(z)-as-integral}
  \end{equation}
  it will first be shown that
  \begin{equation}
    W(w)=\sum_{j,k=0}^{\oo}\varphi_{j,k}(r^{2})r^{j+k}Y_{j,k}.
    \label{eq:v(z)-as-summation}
  \end{equation}
  The calculations will be in the spirit of \cite{MR2925386}.

  It what follows, set
  \begin{align*}
    K(t)&=|1-rt|^{-2s}\\
    m&=\min\left\{ j,k \right\}\\
    T&=r\sqrt{\frac{1+t}{2}}.
  \end{align*}
  By the CR Funk-Hecke formula in Proposition \ref{prop:frank-lieb-funk-hecke-formula},
  \begin{align*}
    W(w)&=\frac{\Gamma^{2}(s)}{n!\Gamma(2\g)}\varphi^{2\g}\int_{\S^{2n+1}}|1-w\cdot\bar\xi|^{-2s}f(\xi)d\sigma\\
    &=\frac{\Gamma^{2}(s)}{n!\Gamma(2\g)}\varphi^{2\g}\sum_{j,k=0}^{\oo}\int_{\S^{2n+1}}|1-w\cdot\bar\xi|^{-2s}Y_{j,k}d\sigma\\
    &=\frac{\Gamma^{2}(s)}{n!\Gamma(2\g)}\varphi^{2\g}\sum_{j,k=0}^{\oo}\int_{\S^{2n+1}}K\left( \frac{w}{r}\cdot\bar\xi \right)Y_{j,k}d\sigma\\
    &=\frac{\Gamma^{2}(s)}{n!\Gamma(2\g)}\varphi^{2\g}\sum_{j,k=0}^{\oo}Y_{j,k}\frac{m!n!}{2\pi2^{n+\frac{|j-k|}{2}}(m+n-1)!}\int_{-1}^{1}dt(1-t)^{n-1}(1+t)^{\frac{|j-k|}{2}}P_{m}^{(n-1,|j-k|)}(t)\\
    &\times\int_{-\pi}^{\pi}d\vartheta{K\left( e^{-i\vartheta}\sqrt{\frac{1+t}{2}} \right)}e^{i(j-k)\vartheta}\\
    &=\frac{\Gamma^{2}(s)}{\Gamma(2\g)}\varphi^{2\g}\sum_{j,k=0}^{\oo}Y_{j,k}\frac{m!}{2\pi2^{n+\frac{|j-k|}{2}}(m+n-1)!}\int_{-1}^{1}dt(1-t)^{n-1}(1+t)^{\frac{|j-k|}{2}}P_{m}^{(n-1,|j-k|)}(t)\\
    &\times\int_{-\pi}^{\pi}d\vartheta|1-2T\cos\vartheta+T^{2}|^{-s}e^{i(j-k)\vartheta}.
  \end{align*}
  By the cosine integral identity \eqref{eq:cosine-integral} we obtain
  \begin{align*}
    W(w)&=\frac{\Gamma^{2}(s)}{\Gamma(2\g)}\varphi^{2\g}\sum_{j,k=0}^{\oo}Y_{j,k}\frac{m!}{2\pi2^{n+\frac{|j-k|}{2}}(m+n-1)!}\int_{-1}^{1}dt(1-t)^{n-1}(1+t)^{\frac{|j-k|}{2}}P_{m}^{(n-1,|j-k|)}(t)\\
    &\times\frac{2\pi}{\Gamma^{2}(s)}\sum_{\mu=0}^{\oo}T^{|j-k|+2\mu}\frac{\Gamma(s+\mu)\Gamma(s+|j-k|+\mu)}{\mu!(|j-k|+\mu)!}\\
    &=\frac{\varphi^{2\g}}{\Gamma(2\g)}\sum_{j,k=0}^{\oo}Y_{j,k}\sum_{\mu=0}^{\oo}\frac{m!}{2^{n+\frac{|j-k|}{2}}(m+n-1)!}\frac{\Gamma(s+\mu)\Gamma(s+|j-k|+\mu)}{\mu!(|j-k|+\mu)!}r^{|j-k|+2\mu}\\
    &\times\int_{-1}^{1}dt(1+t)^{\frac{|j-k|}{2}+\mu}2^{-\frac{|j-k|}{2}-\mu}(1-t)^{n-1}(1+t)^{\frac{|j-k|}{2}}P_{m}^{(n-1,|j-k|)}(t)\\
    &=\frac{\varphi^{2\g}}{\Gamma(2\g)}\sum_{j,k=0}^{\oo}Y_{j,k}\sum_{\mu=0}^{\oo}\frac{m!}{2^{\mu+|j-k|+n}(m+n-1)!}\frac{\Gamma(s+\mu)\Gamma(s+|j-k|+\mu)}{\mu!(|j-k|+\mu)!}r^{|j-k|+2\mu}\\
    &\times\int_{-1}^{1}dt(1-t)^{n-1}(1+t)^{|j-k|+\mu}P_{m}^{(n-1,|j-k|)}(t).
  \end{align*}
  Next, by the Jacobi polynomial integral identity \eqref{eq:jacobi-integral}, we obtain
  \begin{align*}
    W(w) &=\frac{\varphi^{2\g}}{\Gamma(2\g)}\sum_{j,k=0}^{\oo}Y_{j,k}\sum_{\mu=m}^{\oo}\frac{m!}{2^{\mu+|j-k|+n}(m+n-1)!}\frac{\Gamma(s+\mu)\Gamma(s+|j-k|+\mu)}{\mu!(|j-k|+\mu)!}r^{|j-k|+2\mu}\\
    &\times2^{|j-k|+n+\mu}\frac{\mu!}{m!(\mu-m)!}\frac{(|j-k|+\mu)!(m+n-1)!}{(|j-k|+m+n+\mu)!}\\
    &=\frac{\varphi^{2\g}}{\Gamma(2\g)}\sum_{j,k=0}^{\oo}\sum_{\mu=m}^{\oo}r^{|j-k|+2\mu}\frac{\Gamma(s+\mu)\Gamma(s+|j-k|+\mu)}{(\mu-m)!(|j-k|+m+n+\mu)!}\\
    &=\frac{\varphi^{2\g}}{\Gamma(2\g)}\sum_{j,k=0}^{\oo}r^{j+k}Y_{j,k}\sum_{\mu=0}^{\oo}\frac{\Gamma(j+s+\mu)\Gamma(k+s+\mu)}{\Gamma(j+k+n+1+\mu)}\frac{r^{2\mu}}{\mu!}.
  \end{align*}

  Next, by \eqref{eq:hypergeoemtric-definition}, we have
  \begin{align*}
    W(w)&=\frac{\varphi^{2\g}}{\Gamma(2\g)}\sum_{j,k=0}^{\oo}r^{j+k}Y_{j,k}\frac{\Gamma(j+s)\Gamma(k+s)}{\Gamma(j+k+n+1)}F(j+s,k+s,j+k+n+1;r^{2}).
  \end{align*}
  Lastly, by \eqref{eq:hypergeometric-identities-list}, we have
  \begin{align*}
   W(w) &=\frac{\varphi^{2\g}}{\Gamma(2\g)}\sum_{j,k=0}^{\oo}r^{j+k}Y_{j,k}\frac{\Gamma(j+s)\Gamma(k+s)}{\Gamma(j+k+n+1)}\varphi^{n+1-2s}\\
    &\qquad\qquad\qquad\times{F(k+n+1-s,j+n+1-s,j+k+n+1;r^{2})}\\
    &=\sum_{j,k=0}^{\oo}r^{j+k}Y_{j,k}\frac{\Gamma(j+s)\Gamma(k+s)}{\Gamma(j+k+n+1)\Gamma(2\g)}\\
    &\qquad\qquad\qquad\times{F(k+n+1-s,j+n+1-s,j+k+n+1;r^{2})}\\
    &=\sum_{j,k=0}^{\oo}\varphi_{j,k}(r^{2})r^{j+k}Y_{j,k},
  \end{align*}
  thus proving \eqref{eq:v(z)-as-summation}.

  It will now be shown that $u(w)=\varphi^{n+1-s}W(w)$ is a solution.
  Let $g(r)=(1-r)^{n+1-s}\varphi_{j,k}(r)$, and, with an abuse of notation, let $\varphi=1-r$, and $\varphi_{j,k}=\varphi_{j,k}(r)$.
  By the scattering problem \eqref{eq:scattering-problem-on-complex-hyperbolic-space} (see \cite[Theorem 3.10]{MR2472889}), it is sufficient to show
  \[
    \Delta_{\varphi}\left( g(r^{2})r^{j+k}Y_{j,k} \right)-s(n+1-s)g(r^{2})r^{j+k}Y_{j,k}=0,\quad{j,k\geq0}.
  \]

  The following list of computations will be useful:
  \begin{align*}
    g'(r)&=(\varphi^{n+1-s}\varphi_{j,k})'=-(n+1-s)\varphi^{n+1-s-1}\varphi_{j,k}+\varphi^{n+1-s}\varphi_{j,k}'\\
    g''(r)&=(-(n+1-s)\varphi^{n+1-s-1}\varphi_{j,k})'+(\varphi^{n+1-s}\varphi_{j,k}')'\\
    &=(n+1-s)(n+1-s-1)\varphi^{n+1-s-2}\varphi_{j,k}\\
    &-2(n+1-s)\varphi^{n+1-s-1}\varphi_{j,k}'+\varphi^{n+1-s}\varphi_{j,k}''\\
    r(1-r)^{2}g''(r)&=(n+1-s)(n+1-s-1)r\varphi^{n+1-s}\varphi_{j,k}\\
    &-2(n+1-s)r\varphi^{n+1-s+1}\varphi_{j,k}'+r\varphi^{n+1-s+2}\varphi_{j,k}''\\
    (j+k+n+1)(1-r)g'(r)&=-(n+1-s)(j+k+n+1)\varphi^{n+1-s}\varphi_{j,k}\\
    &+(j+k+n+1)\varphi^{n+1-s+1}\varphi_{j,k}'\\
    -(j+k+1)r(1-r)g'(r)&=(n+1-s)(j+k+1)r\varphi^{n+1-s}\varphi_{j,k}-(j+k+1)r\varphi^{n+1-s+1}\varphi_{j,k}'\\
    -jk(1-r)g(r)&=-jk\varphi^{n+1-s+1}\varphi_{j,k}
  \end{align*}
  Now evaluate these computations at $r^{2}$, and resume letting $\varphi=1-r^{2}$ and $\varphi_{j,k}=\varphi_{j,k}(r^{2})$.
  In the computation \eqref{eq:laplace-on-g(r)rY} for the $g(r^{2})$ at hand, it will be convenient to first collect the respective $\varphi_{j,k}$, $\varphi_{j,k}'$, and $\varphi_{j,k}''$ coefficients:
  \begin{align*}
    \varphi_{j,k}\text{ terms: }&
    \begin{cases}
      (n+1-s)(n+1-s-1)r^{2}\varphi^{n+1-s}-(n+1-s)(j+k+n+1)\varphi^{n+1-s}\\
      +(n+1-s)(j+k+1)r^{2}\varphi^{n+1-s}-jk\varphi^{n+1-s+1}
    \end{cases}\\
    \varphi_{j,k}'\text{ terms: }&
    \begin{cases}
      -2(n+1-s)r^{2}\varphi^{n+1-s+1}+(j+k+n+1)\varphi^{n+1-s+1}-(j+k+1)r^{2}\varphi^{n+1-s+1}.
    \end{cases}\\
    \varphi_{j,k}''\text{ terms: }&
    \begin{cases}
      r^{2}\varphi^{n+1-s+2}
    \end{cases}
  \end{align*}
  Putting this all together into \eqref{eq:laplace-on-g(r)rY}, one has
  \begin{equation*}
    \begin{aligned}
      -\frac{1}{r^{j+k}Y_{j,k}\varphi^{n+1-s}}&\Delta_{\varphi}[g(r^{2})r^{j+k}Y_{j,k}]\\
      &=r^{2}\varphi^{2}\varphi_{j,k}''+\left( j+k+n+1-[2(n+1-s)+j+k+1]r^{2} \right)\varphi\varphi_{j,k}'\\
      &+(-(n+1-s)(j+k+n+1)+(n+1-s)(j+k+n+1-s)r^{2})\varphi_{j,k}\\
      &-jk\varphi\varphi_{j,k}
    \end{aligned}
  \end{equation*}
  Now, if
  \begin{align*}
      a&=-(n+1-s)(j+k+n+1)\\
      b&=(n+1-s)(j+k+n+1-s)\\
      c&=jk\\
     y&=\varphi_{j,k}
  \end{align*}
  then
  \begin{align*}
    (a+br^{2})y-c(1-r^{2})y&=(a+b-b+br^{2})y-c(1-r^{2})y\\
    &=(a+b)y-b(1-r^{2})y-c(1-r^{2})y\\
    &=(a+b)y-(b+c)(1-r^{2})y\\
    (a+b)&=-s(n+1-s)\\
    (b+c)&=(n+1-s)(n+1-s-1+j+k+1)+jk\\
    &=(n+1-s)(n-s+j+k+1)+jk\\
    &=n^{2}-2ns+nj+nk+2n-2s+j+k+1-sj-sk+jk+s^{2}\\
    &=(j+n+1-s)(k+n+1-s),
  \end{align*}
  and therefore,
  \begin{align*}
    -\frac{1}{Y_{j,k}r^{j+k}\varphi^{n+1-s}}\Delta_{\varphi}\left[ g(r^{2})r^{j+k}Y_{j,k} \right]&=\varphi\left(r^{2}(1-r^{2})\varphi_{j,k}''+(j+k+n+1-[j+k+2n-2s+3]r^{2})\varphi_{j,k}'\right.\\
    &\left.-(j+n+1-s)(k+n+1-s)\varphi_{j,k}\right)-s(n+1-s)\varphi_{j,k}.
  \end{align*}
  Next, since $\varphi_{j,k}$ is a hypergeometric function, it satisfies the following hypergeometric differential equation:
  \begin{equation*}
    r(1-r)y''+(j+k+n+1-\left[ j+k+2n-2s+3 \right]r)y'-(j+n+1-s)(k+n+1-s)y=0,
  \end{equation*}
  and so the desired equality
  \[
    \Delta_{\varphi}[\varphi^{n+1-s}(w)\varphi_{j,k}(r^{2})r^{j+k}Y_{j,k}]=s(n+1-s)\varphi^{n+1-s}(w)\varphi_{j,k}(r^{2})r^{j+k}Y_{j,k}
  \]
  holds.
\end{proof}

\end{document}